\documentclass[11pt,a4paper,reqno]{amsart}
\usepackage{a4wide}
\usepackage{amsmath}
\usepackage{amsthm}
\usepackage{amsfonts}
\usepackage{amssymb}
\usepackage{amsbsy}
\usepackage{mathrsfs}
\newtheorem{theorem}{Theorem}[section]
\newtheorem{lemma}[theorem]{Lemma}

\newtheorem{proposition}[theorem]{Proposition}
\newtheorem{claim}{Claim}

\theoremstyle{remark}\newtheorem{remark}[theorem]{\it \bf{Remark}\/}

\numberwithin{equation}{section}
\catcode`@=11
\def\section{\@startsection{section}{1}%
  \z@{1.5\linespacing\@plus\linespacing}{.5\linespacing}%
  {\normalfont\bfseries\large\centering}}
\catcode`@=12
\newcommand{\be}{\begin{equation}}
\newcommand{\ee}{\end{equation}}
\newcommand{\bea}{\begin{eqnarray}}
\newcommand{\eea}{\end{eqnarray}}
\newcommand{\bee}{\begin{eqnarray*}}
\newcommand{\eee}{\end{eqnarray*}}

\def\NN{\mathbb{N}}

\def\calA{{\mathcal A}}

\def\calD{{\mathcal D}}

\def\calT{{\mathcal T}}
\def\calN{{\mathcal N}}

\catcode`@=11
\def\supess{\mathop{\operator@font Sup\,ess}}
\catcode`@=12

\def\NN{\mathbb{N}}

\def\RR{\mathbb{R}}

\def\e{\varepsilon}

\def\R2+{\RR ^2_+}

\def\calA{{\cal A}}

\def\calD{{\cal D}}

\def\lim{\mathop{\rm lim}}

\def\supp{{\rm supp}~}
\def\sup{\mathop{\rm sup}}

\def\l{\lambda}

\def\log{{\rm log}}

\def\cal{\mathcal}

\def\pbt{\ti{P}_b}

\def\pbt{\widetilde \psi}
\def\vbt{\widetilde \varphi}

\newcommand{\p}{\partial}
\newcommand{\I}{\infty}
\newcommand{\al}{\alpha}
\newcommand{\ga}{\gamma}
\newcommand{\de}{\delta}
\newcommand{\la}{\lambda} \newcommand{\La}{\Lambda}
\newcommand{\si}{\sigma}
\newcommand{\fy}{\varphi}
\newcommand{\ze}{\zeta}
\newcommand{\om}{\omega}
\newcommand{\Om}{\Omega}
\newcommand{\calI}{\mathcal{I}}
\newcommand{\calL}{\mathcal{L}}

\newcommand{\frkQ}{\mathfrak{Q}}
\newcommand{\ti}{\widetilde}
\newcommand{\lec}{\,\lesssim\,}
\newcommand{\weak}{\operatorname{w-}}

\newcommand{\ip}[2]{( #1 , #2)}
\newcommand{\EQ}[1]{\begin{equation}\begin{split} #1 \end{split}\end{equation}}
\newcommand{\mat}[1]{\begin{pmatrix} #1 \end{pmatrix}}

\newcommand{\Lw}[2]{\{#1,#2\}}

\title[Threshold manifold for critical gKdV]{Codimension one threshold manifold for the critical gKdV equation}
\author[Y. Martel]{Yvan Martel}
\address{Ecole polytechnique, CMLS CNRS UMR7640, 91128 Palaiseau, France}
\email{yvan.martel@polytechnique.edu}
\author[F. Merle]{Frank Merle}
\address{Universit\'e de Cergy Pontoise and Institut des Hautes \'Etudes Scientifiques, AGM CNRS UMR8088, 95302 Cergy-Pontoise, France}
\email{merle@math.u-cergy.fr}
\author[K. Nakanishi]{Kenji Nakanishi}
\address{Department of Mathematics, Kyoto University, Kyoto 606-8502, Japan} 
\email{n-kenji@math.kyoto-u.ac.jp}
\author[P. Rapha\"el]{Pierre Rapha\"el}
\address{Universit\'e de Nice Sophia-Antipolis, Laboratoire J.A. Dieudonn\'e
 CNRS UMR7351, 06108 NICE Cedex 02, France}
\email{pierre.raphael@unice.fr}
\begin{document}

\begin{abstract}
We construct the ``threshold manifold'' near the soliton for the mass critical gKdV equation, completing results obtained in
 \cite{MMR1} and  \cite{MMR2}. In a neighborhood of the soliton, this $C^1$ manifold of codimension one separates solutions blowing up in finite time and solutions in the ``exit regime''. On the manifold, solutions are global in time and converge locally to a soliton. In particular, the soliton behavior is strongly unstable by blowup.
\end{abstract}

\maketitle

\section{Introduction}
\subsection{General setting}
We consider the mass critical generalized Korteweg--de Vries equation:
\begin{equation}\label{kdv}
{\rm (gKdV)}\quad  \left\{\begin{array}{ll}
 u_t + (u_{xx} + u^5)_x =0, \quad & (t,x)\in [0,T)\times\RR, \\
 u(0,x)= u_0(x), & x\in {\mathbb R}.
\end{array}
\right.
\end{equation}
The Cauchy problem is locally well posed in the energy space $H^1(\RR)$ from Kenig, Ponce and Vega \cite{KPV,KPV2}: given $u_0 \in H^1$, there exists a unique\footnote{in a certain sense, e.g., $u\in L^5_xL^{10}_t$ locally in time is sufficient for the uniqueness.} maximal solution $u(t)$ of \eqref{kdv} in $C([0,T), H^1)$ and 
\be
\label{blowucifoi}
T<+\infty \ \  \mbox{implies} \ \ \lim_{t\to T} \|u_x(t)\|_{L^2} = +\infty.
\ee
Moreover, $H^1$ solutions satisfy the conservation of mass and energy: $$M(u(t))=\int u^2(t)= M_0, \ \ E(u(t))= \frac 12 \int u_x^2(t) - \frac 16 \int u^6(t)= E_0.$$ 
Equation \eqref{kdv} satisfies the following symmetries : if $u(t,x)$ satisfies \eqref{kdv}, then, for all $(\l_0,x_0,t_0)\in (0,+\infty)\times\RR\times\RR$, 
$\pm\lambda_0^{\frac12}u(\lambda_0^3 (t-t_0),\lambda_0 (x -x_0))$
also satisfies \eqref{kdv}.

Recall that the traveling wave solutions of \eqref{kdv} have (up to the above symmetries) the following form  $$u(t,x)=Q(x-t)$$ where $Q$ is the ground state solitary wave $$Q(x) =   \left(\frac {3}{\cosh^{2}\left( 2 x\right)}\right)^{\frac14},\quad
Q''+Q^5=Q.$$ 
Recall also the sharp Gagliardo-Nirenberg inequality, \cite{W1983}:
 \be\label{gn}
\forall v\in H^1,\quad
\int \frac{v^6}{6} \leq \int \frac{v_x^2}{2} \left(\frac {\int v^2}{\int Q^2}\right)^2.
\ee
From this inequality and 
the conservation of mass and energy,   $H^1$ initial data with subcritical mass $\|u_0\|_{L^2}<\|Q\|_{L^2}$ generate bounded (in $H^1$) and thus global solutions.
 
The study of singularity formation (existence and behavior of blow up solutions) for $H^1$ initial data 
with mass slightly above the minimal mass
\begin{equation}
\label{utwosmall}
\|Q\|_{L^2}\leq\|u_0\|_{L^2}<\|Q\|_{L^2}+\alpha^* \quad \hbox{for}\quad \alpha^*\ll1, 
\end{equation} 
was initiated in \cite{MMjmpa,MMgafa,Mjams,MMannals, MMjams} and then continued in more recent works \cite{MMR1,MMR2,MMR3}. 

Now, we recall the main result from \cite{MMR1,MMR2}.
Define the two dimensional soliton manifold
$$
\mathcal Q = \left\{ \frac 1{\lambda_0^{\frac 12}} Q\left(\frac {.-x_0}{\lambda_0} \right); \quad
\lambda_0>0, \ x_0\in \RR\right\},
$$ 
and  the $L^2$  tube around $\mathcal Q$ of size $\alpha^*>0$,
\be\label{tube}
\mathcal T_{\alpha^*}=\left\{u\in H^1\ \mbox{such that} \ \inf_{v\in \mathcal Q}
\|u- v\|_{L^2} <\alpha^*\right\}.\ee

Consider the following set of initial data, for $\alpha_0>0$,
$$
\mathcal{A}_Q=
\left\{
u_0=Q+\e   \hbox{ with } \|\e\|_{H^1}<\alpha_0 \hbox{ and }
\int_{x>0} x^{10}\e^2(x)dx< 1
\right\}.$$

\begin{theorem}[Classification of the dynamics in $\mathcal{A}_Q$ \cite{MMR1}]
\label{th:3}
There exists $\alpha^*>0$ small so that for all 
$0<\alpha_0\ll \alpha^*\ll1$, 
for all $u_0 \in \mathcal{A}_Q$, the corresponding solution $u(t)$ of \eqref{kdv} satisfies one of the following
\medskip

\noindent{\em (Blow up)}  For all $t\in [0,T),$ $u(t)\in \mathcal{T}_{\alpha^*}$ and the solution blows up in finite time $T<+\infty$ with the universal blow up behavior
\be\label{blow}
\|u_x(t)\|_{L^2} = \frac {\|Q'\|_{L^2}+o(1)}{\ell_0 (T-t)} \quad 
\hbox{as $t\to T$,}
\ee
where $\ell_0=\ell_0(u_0)>0$ is a constant.
\medskip

\noindent{\em (Soliton)}  The solution is global, for all $t\geq 0, $ $u(t)\in \mathcal{T}_{\alpha^*}$, and  there exist $w_\infty\in{H^1}$, $\lambda_\infty>0$ and $x_\infty$ such that
\begin{equation}
  |  \lambda_{\infty}-1|+||w_\infty||_{H^1} \to 0 \text{ as $\alpha_0\to 0$,}
\end{equation}
 and
\begin{equation}\label{soliton}
\left\|u(t)- Q_{\lambda_\infty,x_\infty}(\cdot-\lambda_\infty^2 t - x_\infty) - e^{-t \partial_x^3} w_\infty \right\|_{H^1}
\rightarrow 0 \quad \hbox{as $t\to +\infty$}.
\end{equation} 

\medskip

\noindent{\rm (Exit)}  There exists $t^*\in (0,T)$ such that 
$u(t^*)\not \in \mathcal{T}_{\alpha^*}$.
\medskip

Moreover, the set of initial data satisfying {\rm (Blow up)}   and the set of initial data satisfying {\rm (Exit)} are open in $\mathcal{A}_Q$ for the $H^1$ norm.
\end{theorem}
 
\begin{remark}
Theorem \ref{th:3} is proved in \cite{MMR1}, except for statement \eqref{soliton}.
The asymptotic result \eqref{soliton}, more precise than the estimate obtained in \cite{MMR1}, since it states scattering for the residual part, is justified in the proof of Proposition \ref{le:2bis} using \cite{KPV} ; see also Tao \cite{Tjde}, Koch and Marzuola \cite{KMapde} for related results.

 It is also proved in \cite{MMR1}  that   the (Blow up) set  contains  the set $\{u_0\in \mathcal A_Q, \ E(u_0)\leq 0, \ u_0 \neq Q \hbox{ up to invariances}\}$. 
 
We expect all solutions in the (Exit) case to scatter at $+\infty$ (i.e. behavior as a linear solution). However, this is an open question.
\end{remark}
 
\begin{remark}[Exotic blow up rates]\label{th:4}
Let us stress the importance of the decay  assumption on the right   for the initial data in the definition of $\mathcal{A}_Q$ to obtain the classification result. 
Indeed in \cite{MMR3}, $H^1$ solutions arbitrarily close to $\mathcal Q$ blowing up in finite time with non generic blow up rates $1/(T-t)^\sigma$, for all $\sigma>\frac {11}{13}$, as well as global solutions growing up at infinity are constructed. These solutions do not enter the above classification, justifying that some decay assumption is necessary to classify the dynamics around $Q$. 

Recall that in \cite{MMR2}, the (unique up to invariances) minimal mass solution $S(t)$ of \eqref{kdv} was constructed.
Such solution blows up in finite time $T$ (let us take $T>0$) and satisfies $\|S(t)\|_{L^2}=\|Q\|_{L^2}$.
This solution has the stable blow  up rate $\frac 1{T-t}$ but it is also an exotic blow up solution, in the sense that its blow up behavior is unstable: for any $0<\epsilon<1$, $(1-\epsilon)S(0)$ gives rise to a global solution (subcritical mass criterion).  In particular, $S(t)\not\in \mathcal A_Q$.
\end{remark}

\begin{remark}
Solutions in the (Exit) case have a universal behavior at the exit time, related to the minimal mass blow up solution. See \cite{MMR2}.
\end{remark}
\subsection{Main result}
 The aim of the present paper is to prove that the (Soliton) set is a co-dimension one manifold in a neighborhood of $\mathcal Q$, thus separating the (Blow up) set from the (Exit) set.
Let
\begin{align*}
 & \mathcal A_0 = \left\{ \e_0 \in H^1; \ \|\e_0\|_{H^1} < \alpha_0,\
  \ \int_{x>0} x^{10} \e_0^2(x) dx < 1 \right\},
 \\& \calA_0^\perp = \{ \e_0 \in \calA_0; \ip{\e_0}{Q}=0\}
\end{align*}
equipped with the norm of $H^1\cap L^2(x_+^{10} dx)$ (we denote $x_+={\rm max}(0,x)$).
 
 \begin{theorem}[Existence of a threshold manifold]\label{th:manifold}
  There exist $\alpha_0>0$, $C>0$  and a $C^1$ function
 \begin{equation}\label{theq:2}
  \calA_0^{\perp}\ni\ga_0 \mapsto A(\ga_0)\in (-C\alpha_0^2,C\alpha_0^2),
 \end{equation}
 such that for all $\ga_0\in \mathcal{A}_0^\perp$, for all $a_0\in (-C\alpha_0,C\alpha_0)$, the solution of \eqref{kdv} corresponding to the initial data $u_0=(1+a_0)Q+\ga_0$ satisfies 
\begin{itemize}
 \item[-] {\rm (Soliton)}   if  $a_0=A(\ga_0)$;
 \item[-] {\rm (Blow up)}   if  $a_0>A(\ga_0)$;
 \item[-] {\rm (Exit)}   if  $a_0<A(\ga_0)$.
\end{itemize}
In particular, there exist a neighborhood $\mathcal O$ of $\mathcal Q$ in $H^1\cap L^2(x_+^{10} dy)$ and a codimension one $C^1$ manifold $\mathcal M \subset \mathcal O$ containing $\mathcal Q$,  such that
 for all $u_0\in \mathcal O$,  the corresponding solution of \eqref{kdv} is in the (Soliton) regime if and only if $u_0\in \mathcal M$.
 \end{theorem}
 
 \noindent{\bf Comments on the result}
\smallskip 
 
 \noindent\emph{1. Construction of the manifold.}
 The second statement in Theorem \ref{th:manifold} is a consequence of the existence of the function $A$. Indeed, for $\lambda_0>0$, $x_0\in \RR$,   define
  $$
  \mathcal M_{\lambda_0,x_0} = \left\{\frac{1}{\lambda_0^{\frac 12}}\left((1+A(\ga_0))Q+\ga_0\right)\left(\frac{.-x_0}{\lambda_0}\right) \, ; \ \ga_0\in \mathcal A_0^\perp\right\}.
$$
Then, the manifold $\mathcal M$ is defined by 
$$
  \mathcal M = \bigcup_{\lambda_0>0,x_0\in \RR} \mathcal M_{\lambda_0,x_0}.
  $$

\smallskip

\noindent\emph{2. Instability of the (Soliton) case.} An immediate corollary of Theorem \ref{th:manifold} is the strong instability of the (Soliton) case. For other works on instability of soliton behavior by blow up, we refer the reader to \cite{PS}, \cite{BC}, \cite{KM1}, \cite{KM2} and \cite{DR}.

\smallskip

\noindent\emph{3. Regularity of the manifold.} In this paper, we prove $C^{1}$ regularity of the function $A$. Our technique should extend to higher order regularity. Indeed, we believe that   the manifold is $C^p$ for any $p$, in a stronger topology than the one of the space $\mathcal A_0^\perp$.
 
Note also that some weight condition on the initial data such as in the definition of the space $\mathcal A_0^ \perp$ is necessary to obtain a threshold behavior separating the blow up region from the exit region. 
Indeed, let $S(t)$ be the minimal mass solution introduced in Remark \ref{th:4} and \cite{MMR2}.
Taking $u_0=(1-\epsilon) S(0)$, for any $0<\epsilon\ll 1$, the solution is global and in the (Exit) regime. For such initial data, there is no transition  between (Blowup) and (Exit).
 
\smallskip
 
\noindent\emph{4. Previous threshold manifold constructions for nonlinear dispersive PDE.} 
 Bates and Jones \cite{BJ}, constructed invariant manifolds in an abstract setting for nonlinear PDE, by the energy argument, applying it to the nonlinear Klein-Gordon equation.
  Krieger and Schlag \cite{KS1} constructed a center-stable manifold for the 1D super-critical nonlinear Schr\"odinger equation around unstable solitons in a specific topology, by the scattering argument for the residual part. 
Similarly Schlag \cite{Sc1} constructed a center-stable manifold around solitons for the $\dot H^{\frac 12}$-critical 3D nonlinear Schr\"odinger equation in a topology stronger that $\dot H^{\frac 12}$. This result was then improved by Beceanu \cite{Be1,Be2} who constructed the manifold in the $\dot H^{\frac 12}$ topology.
 Nakanishi and Schlag \cite{NS, NS1} considered the case of the nonlinear Klein-Gordon equation in 3D, and classified the dynamics under an energy constraint into several regimes, where the center-stable and center-unstable manifolds are the thresholds between scattering and blow-up. 
 See   \cite{NS2} for similar results for the cubic nonlinear Schr\"odinger in 3D. Krieger, Nakanishi and Schlag \cite{KNS1,KNS2}   considered the case of the energy critical nonlinear wave equation.

In all cases cited above, solitons are exponentially unstable, unlike for the $L^2$-critical case, which is degenerate.
Related results for the $L^2$ critical nonlinear Schr\"odinger equation are due to Bourgain and Wang \cite{BW}, Krieger and Schlag \cite{KS2} and Merle, Raphael and Szeftel \cite{MRS},
but  no construction of a {\it threshold} manifold has been achieved  in that case. 
Theorem~\ref{th:manifold} thus completes the first classification of possible behavior for $t\geq 0$, started in \cite{MMR1}, in the case of a nonlinear (not exponential) instability.
 
\smallskip

\noindent\emph{5. Classification for all time $t\in \RR$.}
A related further question is the classification for all time. Namely, is it possible to construct solutions with any of the three behavior as $t\to -\infty$ and $t\to +\infty$ (in the topology $H^1\cap L^2(|x|^{10}dx)$) ?
For exponential instabilities, it has been shown that all possibilities exist, see in particular the ``nine-set results'' in  \cite{NS} and \cite{KNS2}.
Such question is clearly related to the symmetry of the manifold by the transformation $x\mapsto -x$ since the (gKdV) equation is invariant under the transformation $(t,x)\mapsto (-t,-x)$. However, for the critical (gKdV) equation, such question seems really delicate since the ODE on $\lambda(t)$ characterizing the asymptotic behavior of the solution decomposed as
$$
u(t,x)=\frac{1}{\lambda^{\frac 12}(t)}Q\left(\frac{x-x(t)}{\lambda(t)}\right)+ \hbox{residual term},
$$
is $\dot \lambda =0$ i.e. $\l(t) = \ell_0$. At the main order, the behavior at $t\to+\infty$ depends on the sign of $\ell_0$.
In particular,  the change of behavior between $t\to +\infty$ and $t\to -\infty$ should come from the residual part. For exponential instabilities, all possible behavior can be seen at the level of the ODE, dominated by the linearized unstable mode which is absent in the $L^2$ critical case (see \cite{NS} and \cite{KNS2}).

The minimal mass solution $S(t)$ mentioned in Remark \ref{th:4} blows up in positive time but is global in negative time (more precisely, it is in the (Exit) regime for negative time). However, it does not belong to the space $H^1\cap L^2(|x|^{10}dx)$. We do not know whether there is any such solution in this class. 

\subsection{Notation}
For $\lambda_0>0$ and $x_0\in \RR$, we denote 
$$
f_{(\l_0,x_0)  }(x)=\frac{1}{\lambda_0^{\frac 12}}f\left(\frac{x-x_0}{\lambda_0}\right).
$$
Let
 $$\Lambda f(x)=\frac12f(x)+xf'(x) = -\p_{\la_0=1}f_{(\la_0,0)}(x).$$

We denote the $L^2$ scalar product by: $$\ip{f}{g}=\int_{\RR}f(x)g(x)dx.$$ 

Denote by $L$ the linearized operator close to $Q$
\be
\label{deflplus}
Lf=-f''+f-5Q^4f.
\ee 

For a given generic  small constant $0<\alpha^*\ll1 $,  $\delta(\alpha^*)$ denotes a generic small constant with $$\delta(\alpha^*)\to 0\ \ \mbox{as}\ \ \alpha^*\to 0.$$ 
Throughout the paper, the smallness of $\al_0$ is dominating the other small or large parameters, such as $\si$ and $B$ used for the exponential rate on the left. In other words, the parameter $\al_0$ should be chosen in the end after fixing the other parameters, such that all the smallness requirements depending on the other parameters are fulfilled. 
Under this convention, the dependence of $\delta(\al_0)$ on the other parameters is often ignored. 

The variables $(s,y)$ denote the rescaled time and space, where the soliton is renormalized to fixed size and position, while $(t,x)$ denote the original space-time. 

The weighted $L^p$ norm with an exponential weight on the left and a polynomial weight on the right is denoted by 
\EQ{
 \|f\|_{L^p\Lw{\si}{k}} := \|f(x)w(\si,k,x)\|_{L^p_x(\RR)},
 \quad w(\si,k,x):=\begin{cases} e^{\si x} &(x \le 0),\\ (1+x)^k &(x\ge 0), \end{cases}}
for any $\si,k\in\RR$, and similarly the weighted Sobolev norm 
\EQ{
 \|f\|_{H^1\Lw{\si}{k}}^2 := \|f\|_{L^2\Lw{\si}{k}}^2+\|f_x\|_{L^2\Lw{\si}{k}}^2.}
The following weighted $L^2$ norm is frequently used to dispose of localized terms 
\EQ{
 \|f\|_{L^2_{loc}}^2 := \int_\RR |f(x)|^2e^{-\frac{|x|}{10}}dx.}


\subsection{Sketch of the proof of Theorem \ref{th:manifold}}
The general strategy is to construct  directly and explicitly  the  manifold, 
adapting the robust energy-Virial functional introduced in \cite{MMR1} and more standard energy type arguments (\cite{Kato}).

\medskip

In this paper, we only consider solutions of \eqref{kdv} in the (Soliton) regime of Theorem \ref{th:3}, i.e. global and bounded solutions which remain close to the soliton for all $t\geq 0$.

\medskip

{\it -- Decomposition and refined decay estimates.}
Throughout the proof, we decompose  such solution in the following way
\begin{equation}\label{decSk}
	u(t,x)=\mu^{-\frac 12}(t) (Q+\eta)\left(\frac{x-z(t)}{\mu(t)}\right),
\end{equation}
where $\eta(0)=u(0)-Q$, $\mu(0)=1$, $z(0)=0$ (no modulation of the initial configuration).
Setting $s=\int_0^t \mu(t')^{-3} dt',$ $y=\mu(t) x + z(t)$, the function $\eta(s,y)$ satisfies
$$
\partial_s \eta = \partial_y (L \eta - R(\eta)) + {\rm Mod}(\eta),
$$
where $R(\eta) = (Q+\eta)^5 - Q^5-5 Q^4 \eta$ contains nonlinear terms in $\eta$ and ${\rm Mod}(\eta)$ contains linear and quadratic terms in $\eta$ related to the choice of the modulation parameters $(\mu,z)$ and thus to the orthogonality conditions imposed on $\eta$, for all $s\geq 0$,
\begin{equation}\label{orthoSk}
	\partial_s (\eta,\Lambda Q) + (\eta,\Lambda Q) =  \partial_s (\eta,\partial_y \Lambda Q)+(\eta,\partial_y \Lambda Q) =0,
\end{equation}
or equivalently
$$
(\eta(s),\Lambda Q) = e^{-s} (\eta(0),\Lambda Q),\quad 
(\eta(s),y\Lambda Q) = e^{-s} (\eta(0),y\Lambda Q).
$$
(See Section \ref{ss:decop} for explanation on this specific choice of orthogonality relations.)
In Section 2, we first improve   estimates on solutions in the (Soliton) regime from \cite{MMR1}, proving in particular, for all $s\geq 0$, the decay estimates
$$
\|\eta(s)\|_{loc}^2+ \|{\rm Mod}(\eta(s))\|_{L^2_{loc}}^2 \leq \delta(\alpha_0) (1+s)^{-7},
$$
as well as estimates on higher order weighted  Sobolev norms of $\eta$.
Such estimates are consequences of results in \cite{MMR1}, combined with energy techniques from  \cite{Kato},
and rely on the initial weighted bound $\int_{y>0} y^{10} \eta^2(0) dy\lesssim 1$.

\medskip

{\it -- Construction of a Lipschitz graph.}
In Section 3, the construction of the Lipschitz map $A$ whose graph is the local manifold $\mathcal M$ follows from two main arguments. 

\smallskip

(a)  Existence. Given $\gamma \in \mathcal A_0^\perp$, the existence of $a_0=a_0(\gamma)$ with $|a_0|\lesssim \|\gamma\|_{H^1}^2$ so that the solution of \eqref{kdv} with initial data $u(0)=(1+a_0)Q+ \gamma$ is in the (Soliton) regime follows from the trichotomy of Theorem \ref{th:3}. Indeed, it is easy to see that for some
$-\alpha_0\lesssim a_0 \ll - \|\gamma\|_{H^1}^2$, $\|u(0)\|_{L^2}<\|Q\|_{L^2}$ and then the solution is in the (Exit) regime. Moreover, for some
$ \|\gamma\|_{H^1}^2 \ll a_0 \lesssim \alpha_0$, $E(u(0))<0$ and then the solution is in the (Blowup) regime.
Thus, given $\gamma \in \mathcal A_0^\perp$, there exists at least a value of $a_0$ with $|a_0|\lesssim \|\gamma\|_{H^1}^2$ so that $u(0)=(1+a_0)Q+ \gamma$ is in the (Soliton) regime. See details in Section 3.1.

\smallskip

(b) Uniqueness and Lipschitz regularity. Let $u_1$ and $u_2$ be two solutions of \eqref{kdv} in the (Soliton) regime, and $\eta_1$, $\eta_2$ defined accordingly from \eqref{decSk}. In particular, 
$\eta_j(0) = a_j Q + \gamma_j$, where $\gamma_j\in \mathcal A_0^\perp$. Let $\tilde \eta= \eta_1-\eta_2$.
Then, $\tilde \eta$ satisfies
\begin{equation}\label{TetaSk}
	\partial_s \tilde \eta = \partial_y (L\tilde \eta - \widetilde R(\eta_1,\eta_2) \tilde \eta) + \widetilde {\rm Mod}(\eta_1,\eta_2) \cdot \tilde \eta,
\end{equation}
where
$ \widetilde R(\eta_1,\eta_2) = \frac {R(\eta_1)-R(\eta_2)}{\eta_1-\eta_2}$ and $\widetilde {\rm Mod}(\eta_1,\eta_2) \cdot \tilde \eta$ is a linear term in $\tilde \eta$ related to modulation (in particular, $\tilde \eta$ satisfies \eqref{orthoSk}). Note that from the previous estimates
$$\|\tilde \eta(s)\|_{loc}^2\lesssim \delta(\alpha_0) (1+s)^{-7}.$$

The main estimates of the paper, stated in Proposition \ref{pr:s4}, say that {\bf  any} solution  $\tilde \eta$ of \eqref{TetaSk} such that
\begin{equation}\label{liminfSk}
	\liminf_{s\to \infty} (\tilde \eta(s),Q)=0,
\end{equation}
satisfies
\begin{equation}\label{unSk}
\forall s\geq 0,\quad |(\tilde \eta(s),Q)|\lesssim \delta(\alpha_0) (1+s)^{-\frac 52} \|\gamma_1-\gamma_2\|_{H^1},
\end{equation}
\begin{equation}\label{deuxSk}
\sup_{s\geq 0} \|\tilde \eta\|_{loc}^2 + \int_0^{+\infty} \|\tilde \eta(s)\|_{loc}^2 ds \lesssim \|\gamma_1-\gamma_2\|_{H^1}^2.
\end{equation}
The proof of \eqref{deuxSk} is similar to the proof of the main energy estimates in \cite{MMR1} and \cite{MMR2}, using a mixed energy-Virial functional. The proof of \eqref{unSk} is based on  special properties of the $Q$ direction in equation \eqref{TetaSk} (related to the fact that $LQ'=0$, $(Q,\Lambda Q)=(Q,Q')=0$), and the previous estimates, which give
\begin{equation}\label{asSk}
\left|\frac d{ds} (\tilde \eta,Q)(s) \right|\lesssim \delta(\alpha_0) (1+s)^{-\frac 72} \|\tilde \eta\|_{L^2_{loc}}. 
\end{equation}
Integrating \eqref{asSk} on $[s,+\infty)$ using \eqref{liminfSk} and \eqref{deuxSk}, we get \eqref{unSk}.
In fact, the proofs of \eqref{unSk} and \eqref{deuxSk} have to be combined since the Virial relation used for the proof of \eqref{deuxSk} requires some control on $|(\tilde \eta,Q)|$ in addition to the orthogonality relations \eqref{orthoSk}.

In particular, we deduce from \eqref{unSk} at $s=0$ that
$$
|a_1-a_2| \lesssim \delta(\alpha_0) \|\gamma_1-\gamma_2\|_{H^1}.
$$
Given $\gamma \in \mathcal A_0^\perp$, this proves uniqueness of the value of $A(\gamma)$ so that the initial condition $u(0)=A(\gamma)Q+\gamma$ implies the (Soliton) regime, as well as the Lipschitz regularity of the map $A$. See a precise statement in Proposition \ref{th:s4} and
a detailed proof in Section 3.2.

\medskip

{\it -- $C^1$ regularity.}
Let 
$\eta$ and $\eta_n$ correspond to the decompositions of  $u$ and $u_n$, solutions of \eqref{kdv} in the (Soliton) regime with initial data $u(0)=(1+A(\gamma)) Q+\gamma$ and $u_n(0)=(1+A(\gamma_n)) Q+\gamma_n$, where $\gamma_n \to \gamma$ in $H^1$ as $n\to +\infty$.
First,  after extracting a subsequence, we obtain, as $ n\to +\infty,$
 $$\frac{A(\gamma_n)-A(\gamma)}{\|\gamma_n-\gamma\|_{H^1}}\to a_0', \quad \frac{\gamma_n-\gamma}{\|\gamma_n-\gamma\|_{H^1}}\rightharpoonup \gamma', 
 \quad \frac{\eta_n-\eta}{\|\gamma_n-\gamma\|_{H^1}} \rightharpoonup \eta',$$
where $\eta'$ is the (unique) weak equation of 
\begin{equation}\label{etapSk}
\partial_s  \eta' = \partial_y (L \eta' - R'(\eta) \eta') +   {\rm Mod}'(\eta)\cdot\eta',
\quad \eta'(0)=a_0' Q+\gamma'.
\end{equation}
Moreover, by \eqref{unSk}, 
\begin{equation}\label{ezeroSk}
\lim_{s\to \infty} (\eta'(s),Q)=0.
\end{equation}

Second, using identical arguments as in the previous step, 
it follows that given $\gamma$ and $\gamma'$,  there exists a unique  $a_0'$ such that the  solution $\eta'$ of \eqref{etapSk} satisfies \eqref{ezeroSk}. We thus obtain uniqueness of the limit $\lim_{n\to \infty} \frac{A(\gamma_n)-A(\gamma)}{\|\gamma_n-\gamma\|_{H^1}}$, independently of the subsequence and of the choice of the sequence $\gamma_n\to \gamma$. It turns out that $A'(\gamma)\gamma'=a'_0$.
Thus, we see that the proof of differentiability of $A$ is related to a rigidity property of the linearized equation \eqref{etapSk}--\eqref{ezeroSk}.
In particular, $A'(0)=0$, since for $\eta\equiv 0$,  any solution $\eta'$ of \eqref{etapSk} satisfies $\frac d{ds} (\eta'(s),Q)=0$ (see \cite{MMjmpa} for  first rigidity results in similar contexts).
 
The proof of  continuity of  $A'$  follows from similar arguments. See Section~4 for details.

Note that the above argument using weak convergence and weak solutions for \eqref{etapSk}  allow us to prove $C^1$
regularity of $A$ only from estimates \eqref{unSk}-\eqref{deuxSk} and their consequences on the linear equation \eqref{etapSk}--\eqref{ezeroSk}, without having to consider second-order difference estimates, i.e.  on the difference of two solutions of \eqref{etapSk} corresponding to different $\eta$.
 
\subsection*{Acknowledgements}  \quad

This work was partly supported by the project ERC 291214 BLOWDISOL. 

\section{Decomposition and estimates for the (Soliton) case}


\subsection{Regularity results due to the weight assumption}
The following weighted $L^\infty$ Sobolev bound is frequently used to bound nonlinear terms. 
\begin{claim} \label{le:Sob} 
Let $f,g:\RR\to[0,\I)$ be measurable and $w:\RR\setminus\{a\}\to[0,\I)$ be $C^1$ for some $a\in[-\I,+\I]$ and for all $x\in\RR\setminus\{a\}$, 
\EQ{
 |w(x)|\le\sqrt{f(x)g(x)}, \quad w'(x)\begin{cases} \ge -f(x) &(x>a),\\ \le +f(x) &(x<a).\end{cases} }
Then for any $\phi\in H^1_{loc}(\RR)$ such that $\liminf_{|x|\to\I}|\phi(x)|^2w(x)=0$, we have 
\EQ{
 \sup_{x\in\RR}|\phi(x)|^2w(x) \le 2\ip{|\phi(x)|^2}{f}+\ip{|\phi'(x)|^2}{g}.}
Moreover, if $w$ is monotone on $\RR$ ($a=\pm\I$), then 
\EQ{
 \sup_{x\in\RR}|\phi(x)|^2w(x) \le 2\sqrt{\ip{|\phi(x)|^2}{f} \ip{|\phi'(x)|^2}{g}}.}
\end{claim}
\begin{proof}
For all $y>a$, we have 
\EQ{
 |\phi(y)|^2w(y) & = \int_y^\I(-|\phi|^2w'-2\phi\phi'w)dx
 \\ &\le \int_a^\I(|\phi|^2f+2|\phi\phi'|\sqrt{fg})dx
 \le \int_a^\I(2|\phi|^2f+|\phi'|^2g)dx.}
If $w'\ge 0$ then we can drop the term $|\phi|^2f$ and apply Schwarz to the other. The estimate on $y<a$ is the same. 
\end{proof}

\begin{lemma}\label{le:regu}
 Let $u(t)$ be an $H^1$ solution of \eqref{kdv} on $[0,T_1]$ such that $u(0)\in L^2\Lw{0}{5}$. 
Then for $0<t<T_1$, $u(t)\in C^9(\RR)$ in $x$, and for all $\sigma>0$, $0\leq p\leq 10$, 
\begin{equation}\label{eq:tse}
\sup_{0\le t\le T_1}t^{\frac{N(p)}{2}}\|\partial_x^p u(t)\|_{L^2\Lw{\sigma}{5-\frac p2}} \le C(\sigma,p,T_1,\|u_0\|_{H^1\cap L^2(x_+^{10}dx)})<\infty,
\end{equation}
for some $N(p)\ge 0$ determined by $p$. 
Moreover, $\p_t^j \p_x^k u(t,x)\in C((0,T_1)\times \RR)$ for $3j+k\le 9$. In particular, $u$ is a classical solution of \eqref{kdv} on $(0,T_1]$.
\end{lemma}

\begin{proof}
This result can be deduced from  Kato \cite{Kato}. Since $u_0\in H^1$ where \eqref{kdv} is locally wellposed, it suffices to prove the a priori bound \eqref{eq:tse} for smooth rapidly decaying solutions. 
For brevity, we denote $u^{(p)}:=\p_x^pu$. 
Let $\phi:\RR\to\RR$ be a $C^\I$ function such that  
\be
 \phi^{(p)}(x) \sim \begin{cases} e^x &(x<0) \\ 1+x^{10-p} &(x>0) \end{cases} \qquad (p=0,1,\dots,10) 
\ee
and $\phi^{(11)}\ge 0$. For example, fix $\tilde\chi\in C^\infty(\RR)$ such that $\tilde\chi(x)=1$ for $x\le-1$ and $\tilde\chi(x)=0$ for $x\ge -1/2$, and put 
\EQ{ \label{defphi}
 \phi(x)=\calI^{11}(e^x\tilde\chi(x)),}
where $\calI$ is the integral operator defined by $\calI f(x)=\int_{-\I}^x f(y)dy$.  For any smooth rapidly decaying solution $u$, and $p=0,1,\dots,10$, we have the weighted $L^2$ identity: 
\EQ{ \label{wL2id}
 \p_t\ip{(u^{(p)})^2}{\phi_\si^{(p)}} &=-3\ip{(u^{(p+1)})^2}{\phi_\si^{(p+1)}}
 +\ip{(u^{(p)})^2}{\phi_\si^{(p+2)}}
  +2\ip{\p_x^{p+1}(u^5)}{u^{(p)}\phi_\si^{(p)}},}
where $\phi_\si(x):=\phi(\si x)$. 
Henceforth, positive constants may depend on $\si,T_1$. Putting 
\EQ{
 E_{p,\si}:=\ip{(u^{(p)})^2}{\phi_{\si}^{(p)}},}
we have the weighted $L^\infty$ bounds from Lemma \ref{le:Sob}
\EQ{ \label{wLIbd}
 & \|u^{(p)}(\phi_\si^{(p+1)})^{\frac12}\|_\infty \lec E_{p,\si}^{\frac 14}E_{p+1,\si}^{\frac 14},
 \quad \|u^{(1)}(\phi_\si^{(2)})^{\frac 14}\|_\infty \lec \|u_x\|_2^{\frac 12}E_{2,\si}^{\frac 14},
 \\& \|u(1+x_+)^{\frac 52}\|_\infty \lec (E_{0,\si}+\|u\|_{L^2}^2)^{\frac 14}\|u_x\|_2^{\frac 12},
 \quad \|u\|_{L^\I} \lec \|u\|_{H^1}. }

Now we prove the following a priori bound by induction on $p=0,1,\dots,10$: For all $\si>0$, 
\EQ{ \label{wL2indct}
 t^{N(p)}e^{-C_0t}E_{p,\si}(t) + \int_0^t s^{N(p)}e^{-C_0s}E_{p+1,\si}(s)ds \le C,}
where $N(p)\ge 0$ is non-decreasing in $p$, while $C_0,C>0$ depend on $p,\si,T_1$. 
It is obvious that \eqref{wL2indct} implies \eqref{eq:tse}. 

\noindent{\bf (i) $p=0$.} In this case, the last term of \eqref{wL2id} is estimated by 
\EQ{
 2\ip{\p_x(u^5)}{u \phi_\si}=-\frac 53\ip{u^6}{\phi_\si'} \lec \|u\|_{H^1}^4E_{0,\si}.}
Plugging this into \eqref{wL2id} and integrating it by $t$, we obtain \eqref{wL2indct} for $p=0$ with $N(p)=0$. 

\noindent{\bf (ii) $p\ge 1$.} We assume \eqref{wL2indct} up to $p-1$. In this case, the last term of \eqref{wL2id} is equal to 
\EQ{
 -2\ip{\p_x^p(u^5)}{u^{(p+1)}\phi_\si^{(p)}+u^{(p)}\phi_\si^{(p+1)}},}
and the nonlinear term in \eqref{wL2id} is expanded
\EQ{ \label{u5expand}
 \p_x^p(u^5)=5u^4u^{(p)}+c_0^p u^3u^{(1)}u^{(p-1)}+\sum_{\substack{\al_1+\cdots+\al_5=p \\ \al_j\le p-2}}c_\al^p u^{(\al_1)}\cdots u^{(\al_5)}}
with some coefficients $c_0^p,c_\al^p\in\RR$. 
The first term appears for all $p\ge 0$, whose contribution in \eqref{wL2id} is estimated using \eqref{wLIbd}
\EQ{
 &\ip{u^4u^{(p)}}{u^{(p+1)}\phi_\si^{(p)}+u^{(p)}\phi_\si^{(p+1)}}
 \\&\le \|u^4(\phi_\si^{(p)})^{\frac 12}(\phi_\si^{(p+1)})^{\frac{-1}2}\|_{L^\I} E_{p,\si}^{\frac 12}E_{p+1,\si}^{\frac 12} + \|u^4 \phi_\si^{(p+1)}(\phi_\si^{(p)})^{-1}\|_{L^\I} E_{p,\si},}
where both the $L^\I$ norms are bounded by using \eqref{wLIbd}
\EQ{
 \|u^4(1+x_+)^{\frac 12}\|_{L^\I} 
 \le \|u(1+x_+)^{\frac 52}\|_{L^\I}^{\frac 15} \|u\|_{L^\I}^{\frac{19}{5}}
 \lec (E_{0,\si}+\|u\|_{L^2}^2)^{\frac{1}{20}}\|u\|_{H^1}^{\frac{39}{10}}.}
The second term of \eqref{u5expand} appears only for $p\ge 2$, whose contribution is estimate by
\EQ{
 &\ip{u^3u^{(1)}u^{(p-1)}}{u^{(p+1)}\phi_\si^{(p)}}
 \\&\le \|u^3 \phi_\si^{(p)}(\phi_\si^{(2)})^{\frac{-1}4}(\phi_{\frac{\si}2}^{(p-1)})^{\frac{-1}2}(\phi_\si^{(p+1)})^{\frac{-1}2}\|_{L^\I}\|u^{(1)}(\phi_\si^{(2)})^{\frac 14}\|_{L^\I} E_{p-1,\frac{\si}2}^{\frac 12}E_{p+1,\si}^{\frac 12}   
 \\&\lec \|u^3(1+x_+)^{-2}\|_{L^\I} \|u_x\|_{L^2}^{\frac 12} E_{2,\si}^{\frac 14} E_{p-1,\frac{\si}2}^{\frac 12}E_{p+1,\si}^{\frac 12} 
 \lec \|u\|_{H^1}^{\frac 72} E_{2,\si}^{\frac 14} E_{p-1,\frac{\si}2}^{\frac 12}E_{p+1,\si}^{\frac 12},}
and similarly for the inner product with $u^{(p)}\phi_\si^{(p+1)}$. 
The third term of \eqref{u5expand} appears only for $p\ge 3$, whose contribution is estimated by 
\EQ{
 &\ip{u^{(\al_1)}\cdots u^{(\al_5)}}{u^{(p+1)}\phi_\si^{(p)}}
 \\&\le \prod_{j=1}^4\|u^{(\al_j)}(\phi_{\frac{\si}5}^{(\al_j+1)})^{\frac 12}\|_{L^\I} E_{\al_5,\frac{\si}{5}}^{\frac 12} E_{p+1,\si}^{\frac 12}
 \lec \prod_{j=1}^4 E_{\al_j+1,\frac{\si}{5}}^{\frac 12} E_{\al_5,\frac{\si}{5}}^{\frac 12} E_{p+1,\si}^{\frac 12},}
and similarly for the inner product with $u^{(p)}\phi_\si^{(p+1)}$. 
Injecting them into \eqref{wL2id} yields
\EQ{ \label{wL2ineqsy}
 \p_t E_{p,\si} + E_{p+1,\si} \le C_0 E_{p,\si} + C_1(p)E_{p-1,\frac{\si}2}^2 + C_2(p)(E_{2,\si}+E_{0,\frac{\si}{5}}^5+\dots+E_{p-1,\frac{\si}{5}}^5),}
where $C_1(p)=0$ for $p\le 1$ and $C_2(p)=0$ for $p\le 2$. 
Then, using the upper bound of $t\le T_1<\I$ as well, we obtain 
\EQ{
 \p_t e^{-C_0t}E_{p,\si} + e^{-C_0t}E_{p+1,\si} \lec t^{-5N(p-1)}, }
and so, choosing $N(p)=\max(5N(p-1),1)$,  
\EQ{
 &\p_t(t^{N(p)}e^{-C_0t}E_{p,\si}) + t^{N(p)}e^{-C_0t}E_{p+1,\si}
  \lec \max(1,t)+ t^{N(p)-1}e^{-C_0t}E_{p,\si},}
where the integral of the last term on $[0,T_1]$ is bounded by the induction hypothesis \eqref{wL2indct} for $p-1$, while the initial data of $t^{N(p)}e^{-C_0t}E_{p,\si}$ is zero. Thus we obtain \eqref{wL2indct} for all $p=0,1,\dots,10$, which implies \eqref{eq:tse}. Working more precisely with \eqref{wL2ineqsy}, we can also obtain \eqref{wL2indct} and \eqref{eq:tse} with 
\EQ{
 N(0)=0,\ N(1)=1,\ N(p)=2\cdot 5^{p-2}\ (p\ge 2).}

Finally, using the equation of $u$, we deduce that $\p_t^k u(t)\in H^{10-3k}_{loc}\subset C^{9-3k}$ for $0<t<T_1$. In particular, the equation \eqref{kdv} is satisfied in the classical sense on $[0,T_1]$.
\end{proof} 

\subsection{Smooth decomposition for the (Soliton) case} \label{ss:decop}
In \cite{MMR1}, a decomposition of solutions close to solitons based on special profiles $Q_b$
has been introduced.
However, the map $b \mapsto Q_b \in L^2$, defined in  \cite{MMR1} (see its definition in\eqref{eq:210}), is not differentiable at $b=0$, since $\frac {\partial Q_b}{\partial b}|_{b=0} \not \in L^2$. The lack of regularity of the map $b\mapsto Q_b$ complicates the use of this parametrization to prove the $C^1$ regularity of the manifold $\mathcal M$. 
Another difficulty in using the coordinate in \cite{MMR1} is that the rescaling and the translation are unbounded on the radiation part, namely 
\EQ{
 u(x) \mapsto u_{(\la,c)}(x)=\la^{\frac{-1}{2}}u\left(\frac{x-c}{\la}\right)}
is not uniformly continuous in $(\la,c)$ on $H^1\cap L^2(x_+^{10}dx)$ (or any standard Sobolev-type space). 
These are the reasons why we use in this paper a different decomposition. 

Let $u$ be a solution of \eqref{kdv} on $0\le t<T$ with initial data of the form
\EQ{ \label{init form}
 u(0)=Q+\eta_0, \quad \eta_0\in\calA_0.}
Decompose the solution $u$ for $0\le t<T$ by putting
\EQ{
 u(t)=(Q+\eta(t))_{(\mu(t),z(t))},\quad \mu(0)=1,\ z(0)=0,\ \eta(0)=\eta_0,}
for some $(\mu,z)\in C([0,T);(0,\I)\times\RR)$ satisfying $\mu(0)=1$ and $z(0)=0$. 
By the same change of the variables as in \cite{MMR1}
\EQ{ \label{tx2sy}
 & (t,x) \mapsto (s,y); \quad s=\int_0^t \mu(t')^{-3}dt',\ y=\mu(t) x + z(t),
 \\& (s,y) \mapsto (t,x); \quad t=\int_0^s \mu(s')^3ds',\ x=\frac{y-z(s)}{\mu(s)},}
the equation \eqref{kdv} in $(t,x)$ is transformed into the following equation in $(s,y)$:
\EQ{ \label{eqeta}
 & \p_s\eta = \p_y(L \eta -R(\eta)) + \vec \Omega \cdot \vec\p(Q+\eta),}
where 
\EQ{ \label{defOmZ}
 & \Om:=\frac{\mu_s}{\mu}, \quad Z:=\frac{z_s}{\mu}-1, \quad \vec \Omega = \mat{\Omega\\ Z}, \quad \vec\p:=\mat{\La \\ \p_y},}
and  
\EQ{ \label{defR}
 R(\eta):=(Q+\eta)^5-Q^5-5Q^4\eta=\sum_{j=2}^5\frac{5!Q^{5-j}\eta^j}{j!(5-j)!}.}

Now we choose the modulation parameter $(\mu,z)$ by orthogonality forcing 
\EQ{ \label{eqforc}
  (\p_s+1)\ip{\eta}{\vec\La Q}=0, \quad \vec\La :=\mat{\La \\ y \La},}
which is an extension of the standard orthogonality: $\ip{\eta(0)}{\vec\La Q}=\p_s\ip{\eta}{\vec\La Q}=0$. The advantage of \eqref{eqforc} is the absence of initial configuration which could cause loss of regularity mentioned above. 
Injecting \eqref{eqeta} into \eqref{eqforc} yields 
\EQ{
 0=M(\eta)\mat{\Om \\ Z} + \ip{\vec\La Q}{\eta+\p_y(L\eta-R(\eta))},}
where $M(\eta)$ is the $2\times 2$ matrix defined by 
\EQ{ \label{defM}
 M(\eta) =\mat{\ip{\La Q}{\La(Q+\eta)} & \ip{\La Q}{\p_y(Q+\eta)} \\ \ip{y\La Q}{\La (Q+\eta)} & \ip{y\La Q}{\p_y(Q+\eta)}} 
 &=\|\La Q\|_{L^2}^2I-\ip{M'}{\eta}, 
 \\& M':=\mat{\La \La Q & \p_y\La Q \\ \La y\La Q & \p_y y \La Q}.}
Here we used that $\ip{y\La Q}{\p_y Q}=\ip{\La Q}{\La Q-Q/2}=\|\La Q\|_{L^2}^2$. 
Hence, as long as $\|\eta\|_{L^2_{loc}}\ll 1$, $M(\eta)$ is invertible and so 
\EQ{ \label{eqOmZ}
 \mat{\Om \\ Z}=-M(\eta)^{-1}\ip{\vec\La Q}{\eta+\p_y(L\eta-R(\eta))}=:\mat{\om(\eta) \\ \ze(\eta)} =: \vec\Om(\eta).}
Thus we obtain an autonomous equation of $\eta$:
\EQ{ \label{eqetaauto}
 \p_s\eta = \p_y(L \eta -R(\eta))+ \vec\Om(\eta)\cdot\vec\p(Q+\eta),}
where $L$, $R$ and $\vec\p$ were defined respectively in \eqref{deflplus}, \eqref{defR}, \eqref{defOmZ}. 

Multiplying \eqref{eqetaauto} with $\vec\La Q$ shows that it is equivalent to the system of \eqref{eqeta} and \eqref{eqforc}, and, via the change of variables \eqref{tx2sy} and \eqref{defOmZ}, to the original equation \eqref{kdv}, both as long as $\|\eta\|_{L^2_{loc}}\ll 1$. 
The equivalence is valid both in the classical sense and in the distribution sense in $y$. 
Starting from \eqref{eqetaauto}, the modulation parameter $(\mu,z)$ is given in terms of $\eta$: 
\EQ{ \label{eta2muz}
 \log\mu(s)=\int_0^s \om(\eta(s'))ds', \quad z(s)=\int_0^s \mu(s')(\ze(\eta(s'))+1)ds'.}

Now we recall the weighted energy introduced in \cite{MMR1}. 
Let  $\varphi,\psi\in C^{\infty}({\mathbb R})$  be such that:
\bea
\label{defphi1bis}
\varphi(y) =\left\{\begin{array}{lll}e^{y}\ \ \mbox{for} \ \   y<-1,\\
 1+y  \ \ \mbox{for}\ \ -\frac 12<y<\frac 12\\ 
 y^2\ \ \mbox{for}\ \ \mbox{for}\ \ y>2
 \end{array}\right., \ \ \varphi'(y) >0, \ \ \forall y\in \RR,
\eea
\bea
\label{defphi2}
\psi(y) =\left\{\begin{array}{ll} e^{2y}\ \ \mbox{for}   \ \ y<-1,\\
 1  \ \ \mbox{for}\ \ y>-\tfrac 12\end{array}\right., \ \ \psi'(y) \geq 0, \ \ \forall y\in \RR,
\eea
For  $B\geq 100$ large enough, let
\be\label{pB}
\psi_B(y)=\psi\left(\frac yB\right), \ \ \varphi_{B}(y)=\varphi \left(\frac yB\right).
\ee
and for any function $f( y)$, let
\be\label{eq:noeta}
{\mathcal{N}_B}[f]   = \int  f_y^2( y) \psi_B(y)dy 
+ \int  f^2( y) \varphi_B(y)   dy.
\ee
The following lemma gathers the estimates on the solution $\eta$ in the (Soliton) regime which we need to construct the manifold. 

\begin{proposition}[Sharp estimates in the (Soliton) regime]\label{le:2bis}
Let $\eta_0\in\calA_0$ and assume that the solution $u$ of \eqref{kdv} with the initial data $u(0)=Q+\eta_0$ satisfies the {\rm(Soliton)} regime. Then the equation \eqref{eqetaauto} with the initial data $\eta(0)=\eta_0$ has a unique global solution $\eta$ satisfying for all $s\ge 0$ 
\be
\label{controle3}
 \|\eta(s)\|_{H^1} +|\mu(s)-1| \lesssim   \delta(\alpha_0),
\ee
where $\mu$ is given by \eqref{eta2muz}.

\noindent{\em (ii) Sharp decay estimates.} For any $B>100$ large enough, if $\al_0>0$ is small enough then
\begin{equation}\label{eq:eta1}
{\mathcal N}_B[\eta(s)]+|\vec\Om(\eta)|^2 
\lesssim \delta(\alpha_0) (1+s)^{-7},
\end{equation}
for all $s\geq 0$, 
where $\vec\Om$ was defined in \eqref{eqOmZ}, 
\be\label{eq:eta2}
 \|\eta\|_{L^\I_s L^2_y\Lw{\frac{1}{B}}{\frac 92}} + \|\eta\|_{L^2_s H^1_y\Lw{\frac{1}{B}}{4}}
 \lec \delta(\al_0). 
\ee

\noindent{\em (iii) Smoothing estimates.} 
For all $T,\si>0$ and $0\le p\le 10$, there is $N(p)\ge 0$ such that 
\EQ{ \label{regeta}
 \sup_{0\le s\le T}s^{\frac{N(p)}{2}}\|\p_y^p \eta(s)\|_{L^2_y\Lw{\si}{0}}<\I.}
Moreover, $\p_s^j\p_y^k\eta(s,y)\in C((0,\I)\times\RR)$ for $3j+k\le 9$. In particular, $u$ is a classical solution of \eqref{eqetaauto}. 

\noindent{\em (iv) Scattering.}  There exist $\lambda_\infty$, $x_\infty$ and $w_\infty\in H^1$ such that
\begin{equation}\label{parprop}
\mu(t)\to \lambda_\infty,\quad z(t)-\lambda^\infty t \to x_\infty,
\end{equation}
and
\begin{equation}\label{scatprop}
\left\|u(t)- Q_{\lambda_\infty,x_\infty}(\cdot-\lambda_\infty^2 t - x_\infty) - e^{-t \partial_x^3} w_\infty \right\|_{H^1}
\rightarrow 0 \quad \hbox{as $t\to +\infty$}.
\end{equation}
\end{proposition}
\begin{remark}
The parameter $B$ can be taken arbitrarily large, but the larger $B$ requires the smaller $\al_0$. This is always the case throughout the paper. In other words, the parameter $\al_0$ is chosen depending on the parameter $B$. Hence within a small factor $\delta(\al_0)$, the dependence on $B$ is often ignored (e.g., in \eqref{eq:eta1} and \eqref{eq:eta2}), even if it blows up as $B\to\I$. 
\end{remark}

\begin{proof}
{\bf Step 1} Basic estimates from \cite{MMR1}.
 From the definition in Theorem \ref{th:3} and results in \cite{MMR1}, a solution $u(t)$ in the (Soliton) regime is global, bounded, and satisfies
$$
\forall t\geq 0, \quad \inf_{x_0\in \RR}\|u(t)-Q(.-x_0)\|_{H^1} \lesssim \delta(\alpha_0).
$$
To state more precise asymptotic results, we recall the decomposition of $u(t)$ adopted in \cite{MMR1}.
We consider a $C^{\infty}$ function  $\chi$   such that $0\leq \chi \leq 1$, $\chi'\geq 0$ on $\mathbb{R}$,
$\chi\equiv 1$ on $[-1,+\infty)$, $\chi\equiv 0$ on $(-\infty,-2]$.
For $b\in \RR$, define
\begin{equation}\label{eq:210}
Q_b(y) = Q(y) + b \chi_b(y)   P (y) 
\quad\hbox{where}\quad
\chi_b(y)= \chi\left(|b|^{\frac 34} {y} \right),
\end{equation}
where $P(y)\in C^\I$ is in the generalized kernel of $\p_yL$, exponentially decaying on the right while $P(-\I)>0$. See \cite[Proposition 2.2]{MMR1} for the precise definition. 

\begin{lemma}[Decomposition around $Q_b$, \cite{MMR1}]\label{le:2}
Let $\eta_0\in\calA_0$ and assume that the solution of \eqref{kdv} corresponding to $u_0=Q+\eta_0$ satisfies the (Soliton) regime. Then, the following holds.

\noindent{\em (i) Decomposition:} There exist unique $ C^1$ functions $({\lambda}, x,{b}):[0,+\infty) \to (0,+\infty)\times \mathbb{R}^2$ such that
\begin{equation}\label{defofeps}
\forall t\geq 0,\quad
\varepsilon(t,y)={\lambda}^{\frac 12}(t) u(t,{\lambda}(t) y + x(t)) - Q_{{b}(t)}(y)
\end{equation}
satisfies
\be
\label{ortho1}
  \ip{\e(t)}{y {\Lambda}   Q} = \ip{\e(t)}{\Lambda Q} = \ip{\e(t)}{Q} = 0,
\ee
and 
\be
\label{controle}
\|\varepsilon(t)\|_{H^1}+ |{b}(t)|+|1-\lambda(t)|\lesssim   \delta(\alpha_0).
\ee
{\rm (ii) Estimate in the (Soliton) case:} for $B$ large enough,
\be\label{bN}
\hbox{for all $t\geq 0$,}\quad 
|b(t)|\lesssim \mathcal{N}_B[\e(t)],
\ee
\be
\label{fdab}
\hbox{for all $0\leq t_1\leq t_2$,}\quad 
\mathcal N_B[\e(t_2)] +\int_{t_1}^{t_2}\left( b^4(t)+ \ip{(\e_y)^2+ \e^2}{\varphi_B'}\right) dt \le C(B) \mathcal N_B[\e(t_1)].\ee
\end{lemma}

The fact that estimate \eqref{bN} holds for all $t\geq 0$ characterizes the (Soliton) case; see the definition of the separation time $t_1^*$ in Proposition 4.1 of \cite{MMR1} (the norm $\mathcal N_1$ used in \cite{MMR1} is not quite the same as $\mathcal N_B$, but since
$\mathcal{N}_1\lesssim \mathcal{N}_B$, \eqref{bN} indeed holds in the soliton case).
Moreover,  \eqref{fdab} is a consequence of \eqref{bN} together with the following general estimate proved in Lemma~4.3 of \cite{MMR1}
 \be
\label{estfondamentale}
\mathcal N_B[\e(t_2)] +\int_{t_1}^{t_2} \left(b^4(t) + \ip{(\e_y)^2+ \e^2}{\varphi_B'} \right) dt  \le C(B)\{\mathcal N_B[\e(t_1)] + |b^3(t_2)|+|b^3(t_1)|\}.
\ee
Note that   estimates in \cite{MMR1} are stated in the rescaled time variable $s$ defined by ${s}(t)=   \int_{0}^{t} \frac {dt'}{{\lambda}^3(t')}$. Using \eqref{controle}, it is clear that
$\frac t 2\leq s(t)\leq 2t$, and thus, for $\alpha_0$ small enough, all estimates can   be written in the original time variable $t$.

\medskip

Now, we translate \eqref{controle} to the new variables $(\eta,\mu,z)$ 
to prove \eqref{eq:eta1} and \eqref{eq:eta2}. 
For a solution $u$ in the (Soliton) regime, we have, at each $s>0$, 
$$
  u= (Q  + \eta)_{(\mu,z)}=
   \frac 1{\lambda^{\frac 12}}(Q+b P\chi_{b}+\varepsilon)
\left(\frac{.-x}{\lambda} \right)
   $$
with orthogonality conditions  \eqref{ortho1} and estimate \eqref{controle}.
In particular, putting $\ti\mu=\mu/\la$ and $\ti z:=(z-x)/\la$, 
\begin{equation}\label{etaeq}
 \eta(y) =
   \ti\mu^{\frac 12}(Q+b P\chi_{b}+\varepsilon)(\ti\mu y+\ti z) - Q(y)  .
\end{equation}
Using the orthogonality forcing \eqref{eqforc} on $\eta(t)$, we obtain, 
\begin{align*}
 O(\al_0)e^{-s} &=\ip{\vec\Lambda Q}{\eta} = M(0)\mat{\ti\mu-1 \\ \ti z} + O((\ti\mu-1)^2+\ti z^2+|b|+\|\varepsilon\|_2)
\end{align*}
and thus, using \eqref{controle}, we obtain 
\begin{equation}\label{compare1}
 \|\eta\|_{L^\I_sH^1_y} + |1-\mu|+|x-z| \lec \delta(\al_0).
\end{equation}
Injecting this information in \eqref{etaeq}, and using \eqref{fdab}, we obtain 
\EQ{ \label{eta loc dec}
 \calN_{\frac B2}[\eta(s_2)] + \int_{s_1}^{s_2}\left(\|\eta_y(s)\|_{L^2_{loc}}^2 + \|\eta(s)\|_{L^2_{loc}}^2 \right)ds \le C(B)\calN_B[\eta(s_1)],}
where the factor $1/2$ in $\calN_{B/2}$ comes from the rescaling. 

\medskip 

{\bf Step 2} Refined weighted estimates in the (Soliton) regime.
Now, we prove \eqref{eq:eta1} and \eqref{eq:eta2}.  
For $k=0,1,\dots,10$ and $B>100$, let 
\EQ{
 w(s,y):=\phi^{(k)}\left(\frac{\mu(s)y}{2B}\right),}
where $\phi$ is defined in \eqref{defphi}. 
Then by \eqref{eqetaauto} we have the weighted $L^2$ identity
\EQ{ \label{wL2eta}
 &\p_s\ip{\eta^2}{w}+\ip{3(\eta_y)^2+\eta^2}{w'}
  =\ip{2F}{\eta w}+\ip{w'''-\ze(\eta)w'-(5Q^4)'w+5Q^4w'}{\eta^2},}
where $F:=\vec\Om(\eta)\cdot\vec\p Q-\p_yR(\eta)$. 
The terms on the right are estimated by 
\EQ{ 
 & |\ip{\vec\Om(\eta)\cdot\vec\p Q}{\eta w}| \lec |\vec\Om(\eta)|\|\eta\|_{L^2_{loc}} \lec \|\eta\|_{L^2_{loc}}^2,
 \\& |\ip{\p_yR(\eta)}{\eta w}| 
  \lec \ip{\eta^2}{e^{\frac{-|y|}{2}}w+\eta^4w'} 
  \lec \|\eta\|_{L^2_{loc}}^2 + \al_0^4\ip{\eta^2}{w'},
 \\& |\ip{w'''-(5Q^4)'w+5Q^4w'}{\eta^2}| \le C\|\eta\|_{L^2_{loc}}^2+\frac15\ip{\eta^2}{w'},
 \\& |\ip{\ze(\eta)w'}{\eta^2}| \le |\ze(\eta)| \ip{\eta^2}{w'} \lec \al_0\ip{\eta^2}{w'},}
where $|w'''|\le w'/5$ was used on the third line. Hence 
\EQ{ \label{wL2etaest}
 \p_s\ip{\eta^2}{w} + \frac12\ip{(\eta_y)^2+\eta^2}{w'} \le C\|\eta\|_{L^2_{loc}}^2.}
Integrating it in $s$ and using \eqref{eta loc dec}, we obtain for $0\le s_1\le s_2$
\EQ{
 \ip{\eta(s_2)^2}{w} + \frac12\int_{s_1}^{s_2} \ip{(\eta_y)^2+\eta^2}{w'} ds
  \le \ip{\eta(s_1)^2}{w} + C(B) \calN_B[\eta(s_1)]. }
Choosing $s_1=0$ with $k=0$ leads to \eqref{eq:eta2}. 
Similarly, putting for $k=0,1,\dots,7$ 
\EQ{
 M_{k,B}(s):=\ip{\eta^2}{\phi^{(k)}(\mu(s)y/(2B))}+\calN_B[\eta],}
and using $\psi_B+\varphi_B\lec w'$ and \eqref{eta loc dec}, 
\EQ{
 M_{k,\frac B2}(s_2) + \frac{1}{2}\int_{s_1}^{s_2}M_{k+1,B}(s)ds \le C(B)M_{k,B}(s_1).}
Iterating this estimate, we obtain 
\EQ{
 \frac{s^k}{k!}M_{k+1,\frac B2}(s) = \int_0^s \frac{(s')^{k-1}}{(k-1)!}M_{k+1,\frac B2}(s)ds'
 &\le \int_0^\I \frac{(s')^{k-1}}{(k-1)!}CM_{k+1,B}(s')ds'
 \\& = C(\calI^*)^k M_{k+1,B}(0) \le C(2C)^k M_{1,B}(0),}
where $C=C(B)$ and $\calI^*$ is the operator defined by $\calI^*f(s)=\int_s^\I f(s')ds'$. In particular, 
\EQ{
 \calN_B[\eta(s)]+|\vec\Om(\eta)|^2 &\lec \min(M_{1,B}(s),M_{8,B}(s)) \lec (1+s)^{-7}M_{1,2B}(0)
 \\& \lec (1+s)^{-7}\|\eta(0)\|_{H^1}^{\frac 15}\|\eta(0)\|_{L^2\Lw{0}{5}}^{\frac 95} \lec (1+s)^{-7}\delta(\al_0).}
Thus we obtain \eqref{eq:eta1} as well as \eqref{eq:eta2}. 
The smoothing estimates (iii) follow immediately from Lemma \ref{le:regu}, because $u=(Q+\eta)_{(\mu,z)}$ is a global solution of \eqref{kdv} satisfying the initial condition of the lemma and $(\mu,z)$ is given by \eqref{eta2muz}. 
\end{proof}

\textbf{Step 3} Scattering on the background of a soliton.
Now, we prove the scattering result \eqref{scatprop}, using the $L^2$ Cauchy theory (and scattering theory for small data) of \cite{KPV, KPV2} and the information obtained in \eqref{eq:eta1}-\eqref{eq:eta2} to prove that the remainder scatters in $L^2_x$ as $t\to\I$. This section is inspired by (and is simpler than) \cite{Tjde} (see also \cite{KMapde}).

First we go back to the original space-time $(t,x)$ by putting $v(t,x):=\eta_{(\mu,z)}(s,y)$. In the following, we abbreviate $\frkQ:=Q_{(\mu(t),z(t))}$. Then we have $u=\frkQ+v$ and 
\EQ{
 & v_t=-v_{xxx}-\p_x(v^5+F)+G, 
 \\& F:=(\frkQ+v)^5-\frkQ^5-v^5, \quad G:=\mu^{-2}\vec\Om(\eta)\cdot\vec\p\frkQ.} 
We rely on the following space-time estimates in \cite{KPV} and \cite{KPV2} on the free propagator $U(t):=e^{-t\p_x^3}$ : 
\EQ{
 & \|U(t)f\|_{L^\I_tL^2_x\cap L^5_xL^{10}_t} \lec \|f\|_{L^2_x},
 \\& \left\|\int_0^t U(t-s)f(s)ds\right\|_{L^\I_tL^2_x \cap L^5_xL^{10}_t} \lec \|f\|_{L^{\frac 54}_xL^{\frac{10}9}_t},
 \\& \left\|\int_0^t \p_xU(t-s)f(s)ds\right\|_{L^\I_tL^2_x \cap L^5_xL^{10}_t} \lec \|f\|_{L^1_xL^2_t}.}
Applying them to the Duhamel formula of $v$, we obtain 
\EQ{ \label{v stri}
 & \|v\|_{L^\I_tL^2_x \cap L^5_xL^{10}_t}
 \lec \|v(0)\|_2 + \|F\|_{L^1_xL^2_t} + \|v^5\|_{L^1_xL^2_t}
  + \|G\|_{L^{\frac 54}_xL^{\frac{10}9}_t},
 \\& \|v_x\|_{L^\I_tL^2_x \cap L^5_xL^{10}_t}
 \lec \|v_x(0)\|_2 + \|F_x\|_{L^1_xL^2_t} + \|v^4v_x\|_{L^1_xL^2_t}
  + \|G\|_{L^1_xL^2_t},}
where the quintic term is bounded by H\"older
\EQ{
 &\|v^5\|_{L^1_xL^2_t} = \|v\|_{L^5_xL^{10}_t}^5, 
 \quad \|v^4v_x\|_{L^1_xL^2_t} \le \|v\|_{L^5_xL^{10}_t}^4\|v_x\|_{L^5_xL^{10}_t}.}
Since $\|v(0)\|_{H^1}\ll 1$, the following bound is sufficient: 
\EQ{
 \|F\|_{W^{1,1}_xL^2_t} + \|G\|_{L^{\frac 54}_xL^{\frac{10}9}_t \cap L^1_xL^2_t} \le \delta(\al_0).}
This follows from the refined estimate \eqref{eq:eta1}. Indeed, putting $f:=|v_x|+|v|$
\EQ{
 \|F\|_{W^{1,1}_xL^2_t} \lec \|(1+|x|)(|\frkQ_x|+|\frkQ|) f\|_{L^2_xL^2_t}
 &\lec \|(1+|z|)e^{\frac{-|x-z|}2}f\|_{L^2_{t,x}}
 \\&\lec \|f\|_{L^\I_{0<t<1}L^2_x} + \|te^{\frac{-|x-z|}2}f\|_{L^2_{t,x}}, }
and the last term is bounded by \eqref{eq:eta1}
\EQ{
 \|te^{\frac{-|x-z|}2}(|v_x|+|v|)\|_{L^2_{t,x}}^2 \le C(B)\int_0^\I s^2\calN_B[\eta(s)]ds \lec \delta(\al_0).}
Similarly we have for any $p,q\in[1,\I]$,
\EQ{ 
 \|G\|_{L^q_xL^p_t} 
 &\lec \|(1+|x|^{2})G\|_{L^p_tL^\I_x} 
 \\&\lec \|\vec\Om(\eta)(1+|z|^{2})e^{\frac{-|x-z|}2}\|_{L^p_tL^\I_x}\lec \delta(\al_0)\|(1+t)^{-\frac 72+2}\|_{L^p_t} \lec \delta(\al_0).}
Plugging these estimates into \eqref{v stri}, we obtain a priori global space-time bound 
\EQ{
 \|v\|_{L^\I_t H^1_x \cap W^{1,5}_x L^{10}_t} \lec \|v(0)\|_{H^1_x} + \delta(\al_0),}
as well as the scattering of $v$, namely the strong convergence in $H^1$ of $e^{t\p_x^3}v(t)$ as $t\to\I$ to some $w_\infty\in H^1$. 
From \eqref{eq:eta1}, it follows easily that there exist  $\lambda_\infty>0$ and $x_\infty$ so that \eqref{parprop} holds.
This implies the scattering statement \eqref{scatprop}. 

\section{Construction of the Lipschitz manifold}
This section is devoted to the proof of the following Proposition \ref{th:s4}.

\begin{proposition}[Construction of a Lipschitz graph]\label{th:s4}
For any $\ga\in\calA_0^\perp$, there exists a unique $A(\ga)\in (-C\alpha_0,C\alpha_0)$ (for some absolute constant $C>0$) with the following properties. Let $u$ be the solution of \eqref{kdv} with initial data $u(0)=(1+a_0)Q+\ga$. 
\begin{itemize}
 \item $u$ is in the {\rm (Blow up)} regime if $A(\ga)<a_0\lec \al_0$;
 \item $u$ is in the {\rm (Soliton)} regime if $a_0=A(\ga)$;
 \item $u$ is in the {\rm (Exit)} regime if $-\al_0\lec a_0<A(\ga)$;
 \item Lipschitz regularity: For $B>100$ large enough, and $\al_0>$ small enough, there exists $C(B)>0$ such that for any $\ga_j\in \calA_0^\perp$ ($j=1,2$),  
 \begin{equation}
  |A(\ga_1) -A(\ga_2)|
  \le C(B)\left\{\|\ga_1\|_{H^1\Lw{\frac{1}{B}}{0}}+\|\ga_2\|_{H^1\Lw{\frac{1}{B}}{0}}\right\}\|\ga_1-\ga_2\|_{H^1\Lw{\frac{1}{B}}{0}}.
 \end{equation}
\end{itemize}
\end{proposition}
The above Lipschitz bound together with the trichotomy implies $|A(\ga)|\lec\|\ga\|_{H^1}^2$. 

\subsection{Existence}
For each $\ga\in \calA_0^\perp$,  
we prove the existence of at least one value of $a_0=A(\ga)=O(\al_0^2)$ 
such that the solution of \eqref{kdv} with initial data $u(0) = (1+a_0)Q + \ga$, is in the {\rm (Soliton)} regime. We use the trichotomy stated in Theorem \ref{th:3} (see \cite{MMR1}).

On the one hand, by the orthogonality we have
\begin{align*}
 \|u(0)\|_{L^2}^2 = (1+a_0)^2\|Q\|_{L^2}^2 + \|\ga\|_{L^2}^2.
\end{align*}
Thus, for $-\al_0\lec a_0 \ll -\|\ga\|_{L^2}^2$, we have $\|u(0)\|_{L^2}<\|Q\|_{L^2}$ and so, the corresponding solution $u$ is   global satisfying the (Exit) regime. 

On the other hand, we have similarly,
\EQ{
\label{energbound}
 E(u(0))&=E(Q)+\ip{E'(Q)}{a_0Q+\ga} + O(\|a_0Q+\ga\|_{H^1}^2)
 \\&=-a_0\|Q\|_{L^2}^2 + O(|a_0|^2+\|\ga\|_{H^1}^2).}
Thus, for $\|\ga\|_{H^1}^2\ll a_0 \lec \al_0$, we have $E(u(0))<0$ and so, by Theorem~1.1~(i) in \cite{MMR1}, the solution $u$ blows up in finite time.

Since the cases (Blowup) and (Exit) are open in $H^1$ (see \cite{MMR1}), there exists at least one value of $a_0=O(\al_0^2)$, such that the solution of \eqref{kdv} with initial data $(1+a_0)Q+\ga$ satisfies the (Soliton) regime.
 

\subsection{Uniqueness and Lipschitz regularity}

In this section, we prove the uniqueness and the Lipschitz regularity of the function $A$. Let $\eta_1,\eta_2$ be two solutions of \eqref{eqetaauto} in the (Soliton) regime, and consider the difference
\begin{equation}\label{tti2}
\ti \eta(s,y):=\eta_1(s,y)-\eta_2(s,y).
\end{equation}
Then it satisfies  
\begin{align}
 \p_s \ti\eta = \p_y[(L - \ti R(\eta_1,\eta_2))\ti\eta] + \ip{\ti \Omega(\eta_1,\eta_2)}{\ti\eta}\cdot \vec\p(Q+ \eta_1)
  + \vec\Om(\eta_2)\cdot\vec\p \ti \eta, \label{eqe}
\end{align}
where $\ti R$ and $\ti\Omega$ are defined such that we have 
\EQ{
 \ti R(\eta_1,\eta_2)\ti\eta=R(\eta_1)-R(\eta_2), \quad \ip{\ti \Omega(\eta_1,\eta_2)}{\ti\eta}:= \vec\Om(\eta_1)-\vec\Om(\eta_2).} 
Explicitly, they are given by 
\EQ{ \label{deftiR}
 \ti R(\eta_1,\eta_2)  &:= \sum_{j=1}^4\sum_{k=0}^{j} \frac{5!Q^{4-j}\eta_1^{j-k}\eta_2^{k}}{(j+1)!(4-j)!},}
\EQ{ \label{deftiOm}
 \\ \ti \Omega(\eta_1,\eta_2) &:= -M(\eta_2)^{-1}(1-L\p_y+\ti R(\eta_1,\eta_2)\p_y)\vec\La Q
 \\ &\quad -\int_1^2 M(\eta_\theta)^{-1}M'M(\eta_\theta)^{-1}M(\eta_1)\vec\Om(\eta_1)d\theta, \quad \eta_\theta:=(2-\theta)\eta_1+(\theta-1)\eta_2,}
 where $M$ and $M'$ are defined in \eqref{defM}.

Denote the projection orthogonal to $Q$ by 
\EQ{
 P_\perp f := f - \|Q\|_{L^2}^{-2}\ip{Q}{f}.}
Now we are ready to state the core estimates in this paper. 
\begin{proposition}\label{pr:s4}
Let $\eta_1,\eta_2$ be two solutions of \eqref{eqetaauto} in the (Soliton) regime. Let $\ti\eta$ be a solution of \eqref{eqe} in the same function spaces as $\eta$ in Lemma \ref{le:2bis}(iii), satisfying 
\EQ{ \label{tietaforc}
 (\p_s+1)\ip{\ti\eta}{\vec\La Q}=0,}
\EQ{ \label{tiadec}
 \liminf_{s\to\I}|\ip{\ti\eta}{Q}|=0.}
Then for $B>100$ large enough, and $\al_0>0$ small enough, there exists $C(B)>0$ such that 
\EQ{ \label{eq:s4}
 |\ip{Q}{\ti\eta(s)}| &\le \begin{cases} C(B)\delta(\al_0)(1+s)^{\frac{-5}2}\|P_\perp\ti\eta(0)\|_{H^1_y\Lw{\frac{1}{B}}{0}},\\ 
  C(B)\left\{\|\eta_1(0)\|_{H^1_y\Lw{\frac{1}{B}}{0}}+\|\eta_2(0)\|_{H^1_y\Lw{\frac{1}{B}}{0}}\right\}\|P_\perp\ti\eta(0)\|_{H^1_y\Lw{\frac{1}{B}}{0}}, \end{cases}}
for all $s\geq 0$, and
\begin{align}
& \|\ti\eta\|_{L^\I_sH^1_y\Lw{\frac{1}{B}}{0} \cap L^2_sH^1_y\Lw{\frac{1}{B}}{\frac{-3}{4}}} \le C(B)\|P_\perp\ti\eta(0)\|_{H^1_y\Lw{\frac{1}{2B}}{0}}. \label{eq:s4tri}
\end{align}
\end{proposition}
Note that $\ti\eta$ is independent of $\eta_1,\eta_2$ in the above proposition. This freedom is needed in order to recycle it later for $C^1$ estimate of the manifold. 
Applying this proposition to $\ti\eta=\eta_1-\eta_2$, 
\eqref{eq:s4} implies the uniqueness of $a=A(\ga)$ in the (Soliton) regime for each $\ga\in\calA_0^\perp$. Then the above existence together with the openness of the other two regimes implies the trichotomy in Proposition \ref{th:s4}. The Lipschitz regularity of $A$ follows from the second estimate in \eqref{eq:s4}. 
The rest of this section is devoted to the proof of Proposition \ref{pr:s4}.

\begin{proof}[Proof of Proposition \ref{pr:s4}] It is similar to the proof in \cite{MMR2} of uniqueness of the minimal mass solution. We set 
\begin{equation}\label{defat}
 \ti a := \|Q\|_{L^2}^{-2}\ip{\ti\eta}{Q}, \quad \eta:=\mat{\eta_1 \\ \eta_2}.
\end{equation}
Then using \eqref{deftiR}, \eqref{deftiOm}, \eqref{eqe} and \eqref{eq:eta1}, we obtain
\EQ{ \label{cl:p1}
 &|\ip{\ti\Om(\eta_1,\eta_2)}{\ti\eta}| \lec \|\ti\eta\|_{L^2_{loc}},  
 \quad |\p_s\ti a| \lec \|\ti\eta\|_{L^2_{loc}} \|\eta\|_{L^2_{loc}} 
  \lec \|\ti\eta\|_{L^2_{loc}}\delta(\al_0)(1+s)^{-\frac 72}.} 
Integrating the latter on $[s,+\infty)$ and using \eqref{tiadec}, 
we obtain, for all $s\geq 0$,
\begin{equation}\label{poura}
|\ti a(s)|\lesssim \delta(\alpha_0) (1+s)^{-\frac 52} \|\ti\eta\|_{L^\I_sL^2_{loc}}.
\end{equation}
The crucial fact for constructing the manifold is that the last factor is controlled by the initial data $\ti\eta(0)$. 
More precisely, we have 
\begin{claim}[Estimates on $\ti \eta$]\label{cl:e}
\begin{align}
&\|\ti\eta\|_{L^\I_s H^1_y\Lw{\frac{1}{B}}{0}}
 + B^{\frac{-1}{2}}\|\ti\eta\|_{L^2_s H^1_y\Lw{\frac{1}{B}}{\frac{-3}{4}}}
 \lec \|\ti \eta(0)\|_{H^1_y\Lw{\frac{1}{2B}}{0}}
 + B^{\frac 12}\|\ti a(s)\|_{L^\I_s\cap L^2_s}.
\label{cl:ee1}
\end{align}
\end{claim}
Combining this and \eqref{poura} yields \eqref{eq:s4tri} as well as the first estimate in \eqref{eq:s4}. Applying \eqref{eq:s4tri} to a solution $\eta$ itself (by setting $\eta_1=\eta=\ti\eta$ and $\eta_2=0$), we obtain 
\EQ{ \label{H1locbd}
 &\|\eta\|_{L^\I_s H^1_y\Lw{\frac{1}{B}}{0}}
 +\|\eta\|_{L^2_s H^1_y\Lw{\frac{1}{B}}{\frac{-3}{4}}}
 \le C(B)\|\eta(0)\|_{H^1_y\Lw{\frac{1}{2B}}{0}}.}
Using these $L^2_s$ bound in the second estimate of \eqref{cl:p1}, we obtain
\EQ{
 \|\p_s\ti a\|_{L^1_s} \lec \|\ti\eta\|_{L^2_sL^2_{loc}}\|\eta\|_{L^2_sL^2_{loc}} \le C(B)\|\ti\eta(0)\|_{H^1_y\Lw{\frac{1}{2B}}{0}} \|\eta(0)\|_{H^1_y\Lw{\frac{1}{2B}}{0}}.}
Integrating this in $s$ yields the second estimate in \eqref{eq:s4}. 

Thus it remains to prove Claim \ref{cl:e}. We rely on the same type of energy functional as in Section 3 of \cite{MMR1} and Proposition 5.1 of \cite{MMR2}, with a slight modification of the weights. 
Let $\hat\fy:\RR\to\RR$ be a smooth function satisfying 
\EQ{
 \begin{cases}\hat\fy \sim e^{-y} &\text{ for }y<0,\\ \hat\fy =1+y &\text{ for }|y|< \frac 12,\\ \hat\fy \sim 1 &\text{ for }y>0, \end{cases} 
 \qquad  \begin{cases} \hat\fy' \sim e^{-y} &\text{ for }y<0\\ \hat\fy' \sim (1+y)^{-\frac 32} &\text{ for }y>0. \end{cases}}
for $B>100$ large enough, we set
\begin{equation}\label{defphit}
 \vbt(s,y) = \hat\fy \left( \frac {\mu_2(s) y}{B}\right),
 \quad \log \mu_2:=\int_0^s \om(\eta_2(s')) ds'.
\end{equation}
For $0<\sigma<\frac 1{10}$ small to be fixed later depending on $B$, consider a smooth function satisfying
\bea
\label{defphi3}
\psi_{\sigma}(y) =\begin{cases} e^{2y} & \mbox{for}   \ \ y<-1,\\
 1 - \frac \sigma {(1+y)^{\frac 32}}  & \mbox{for}\ \ y>-\tfrac 12, \end{cases} \quad \psi_\sigma'(y) > 0 \ \ \forall y\in \RR,
\eea
and we set
\begin{equation}\label{defpsit}
\pbt(y)= \frac{1}{\mu_2^2(s)} \psi_\sigma \left( \frac y {B}\right).
\end{equation}
For simplicity, we denote $\vbt' = \partial_y \vbt = \frac {\mu_2}{B} \hat\fy'(\frac {\mu_2 y}{B})$, and similarly $\ti\psi'=\p_y\ti\psi = \frac{1}{\mu_2^2B}\psi_\si'(\frac {y}{B})$ though $\vbt$ and $\ti\psi$ also depend on $s$ through the function $\mu_2$. 
Let
$$
\ti {\mathcal F} (s)= \int \left(\pbt\left((\ti \eta_y)^2 - 5 Q^4 \ti \eta^2 \right) + \vbt \ti \eta^2 \right)(s,y) dy.
$$
Note that the exact  expressions of $\vbt$ and $\pbt$, in particular, the position and powers of $\mu_2$ in their definitions,  lead to interesting cancellations in \eqref{anotherremark}. These cancellations allow us to work without polynomial weight for $y>0$ in the definition of $\widetilde {\mathcal F}$.
There is another more technical difference with respect to the functions $\varphi,\psi$ defined in \eqref{defphi1bis}, \eqref{defphi2} : since the function $\psi_\sigma$ has a positive derivative on $\RR$, when differentiating $\ti {\mathcal F}$, one gets a global control of  $\ti\eta_{yy}$ in a weighted norm thanks to the local Kato smoothing effect.
This is useful in controlling  some terms in the proof of Claim \ref{cl:e}.  The constant $\sigma>0$ is to be taken small enough in order to preserve at the main order the Virial structure, see \eqref{vir}.

We claim, for some $\theta >0$,
\EQ{
  \frac {d\ti {\mathcal F}} {ds} &+ 
\theta \ip{(\ti \eta_{yy})^2}{\pbt'} + \theta\ip{(\ti \eta_y)^2 + \ti \eta^2}{\vbt'}
 \\& \lec  B\ip{\ti\eta}{Q}^2 + |\ip{\ti\eta}{\vec\La Q}|^2 
  + B\|\ti \eta\|_{L^2_{loc}}^2
\left( \|\eta_1\|_{H^1_y\Lw{\frac 1{4B}}{\frac 52}}^2 + \delta(\alpha_0) (1+s)^{-\frac 72}\right). \label{dFt}}

Assuming \eqref{dFt} for the moment, we finish the proof of \eqref{cl:ee1}.
First, as in Proposition 3.1 of \cite{MMR1}, using $|1-\mu_2(s)|\leq \delta(\alpha_0)$, we note that, for some $\theta_1>0$.
$$ 
  \ti {\mathcal F} \geq
\theta_1  \|\ti\eta\|_{H^1_y\Lw{\frac{1}{B}}{0}}^2   -\frac 1{\theta_1}\ip{\ti\eta}{Q}^2-\frac{1}{\theta_1}|\ip{\ti\eta}{\vec\La Q}|^2 - O(e^{-\frac B2}+\si)\|\ti\eta\|_{L^2_{loc}}^2, 
$$
where the last term is absorbed by the first term on the right, by choosing $B>1$ large enough and $\si>0$ small enough. 
Moreover, since $\mu_2(0)=1$,
$\ti {\mathcal F}(0) \lesssim \|\ti\eta(0)\|_{H^1_y\Lw{\frac{1}{2B}}{0}}^2$. 
Thus, integrating \eqref{dFt} and using \eqref{tietaforc} and \eqref{eq:eta2}, we obtain
\begin{align*}
\|\ti\eta\|_{L^\I_sH^1_y\Lw{\frac{1}{B}}{0}}^2
 + B^{-1}\|\ti\eta\|_{L^2_sH^1_y\Lw{\frac{1}{B}}{\frac{-3}{4}}}^2
 &\lec B|\ti a|_{L^2_s \cap L^\I_s}^2 + \|\ti\eta(0)\|_{H^1_y\Lw{\frac{1}{2B}}{0}}^2
   + B\delta(\al_0)\|\ti \eta\|_{L^\I_sL^2_{loc}}^2 
\end{align*}
The last term is absorbed by the left side, hence we obtain \eqref{cl:ee1}.

\medskip

\noindent{Proof of \eqref{dFt}.} For brevity, we put
\EQ{
 \mat{\Om_2 \\ Z_2}:=\vec\Om(\eta_2), \quad \ti R:=\ti R(\eta_1,\eta_2), \quad \ti\Om:=\ti\Om(\eta_1,\eta_2).}
Using the equation of $\ti \eta$, we have
\begin{align*}
&\frac 12 \frac {d\ti {\mathcal F}} {ds} =
\ip{\ti \eta_s}{-\pbt' \ti \eta_y + \pbt L \ti \eta + (\vbt-\pbt) \ti \eta}  + \frac{\Om_2}2\ip{y\vbt'}{\ti \eta^2} - \Om_2\ip{\ti\psi}{(\ti\eta_y)^2-5Q^4\ti\eta^2}\\
& = \ip{\p_y L\ti\eta}{-\pbt' \ti \eta_y + \pbt L \ti \eta + (\vbt-\pbt) \ti \eta}\\
& + \ip{\ti\Om}{\ti\eta}\cdot\ip{\vec\p(Q+\eta_1)}{-\pbt' \ti \eta_y + \pbt L \ti \eta + (\vbt-\pbt) \ti \eta}\\
& 
+ \Omega_2 \ip{\Lambda \ti \eta}{- (\pbt \ti \eta_y)_y  - 5Q^4 \ti \eta \pbt + \vbt   \ti \eta} + \frac{\Omega_2}{2} \ip{y\vbt'}{\ti \eta^2} - \Om_2\ip{\ti\psi}{(\ti\eta_y)^2-5Q^4\ti\eta^2} \\
& + \ip{Z_2\p_y\ti\eta}{-\pbt' \ti \eta_y + \pbt L \ti \eta + (\vbt-\pbt) \ti \eta} \\
&- \ip{\p_y\ti R\ti\eta}{- (\pbt \ti \eta_y)_y  - 5Q^4 \ti \eta \pbt + \vbt   \ti \eta} \\
& =: f_1 + f_2 + f_3+ f_4+ f_5.
\end{align*}

- Estimate of $f_1$.
We claim that for $\sigma>0$   small enough,   for some small constant $\theta>0$,
\be\label{F1}
 f_1 +\theta \ip{(\ti\eta_{yy})^2}{\pbt'} +\theta \ip{(\ti\eta_y)^2+(\ti \eta)^2}{\vbt'} 
 \lesssim |\ti a|^2 + |\ip{\ti\eta}{\vec\La Q}|^2.
\ee
Indeed, by explicit computations (integrations by parts), we   have
\begin{align*}
 2 f_1
& = -   \int \left( 3 \pbt' (\ti \eta_{yy})^2 + (3 \vbt'+\pbt'-\pbt''') (\ti\eta_y)^2 + (\vbt'-\vbt''') \ti \eta^2\right)\\
&\quad + \int 5 Q^4 \ti\eta^2(\vbt'-\pbt') + \int 20 Q^3 Q' \ti\eta^2 (\pbt-\vbt)\\
&\quad + 10 \int \pbt'\ti \eta_y(4Q'Q^3 \ti\eta + Q^4 \ti\eta_y) 
 + \int \pbt'  ( -2\ti\eta_{yy}+\ti\eta - 5Q^4 \ti\eta)    5Q^4 \ti\eta\\
& = 2 f_1^> + 2 f_1^{\sim },
\end{align*}
where $f_1^>$ corresponds to the region of integration $|y|> \frac B3$ and 
$f_1^{\sim}$ corresponds to  $|y|< \frac B3$.

For $|y|>\frac {B}3$,  $Q(y)$ is small as $B$ is large and thus, for $B$ large enough, we check easily using that $|\pbt'''|\ll \pbt'$, $|\vbt'''|\ll\vbt'$ and $|Q|\ll \min(\pbt',\vbt')$ on $|y|>\frac {B}3$, 
\begin{equation}\label{ygrand}
2 f_1^> \leq - \frac 12 \int_{|y|>\frac {B}3} \ti \psi' (\ti \eta_{yy})^2 + \ti \varphi' \left( (\ti \eta_y)^2 + \ti \eta^2\right).
\end{equation}

For $|y|<\frac {B}3$, note first that  
$$
( \ti \varphi -\ti\psi  )(y) = \frac {\mu_2 y}{B} + 1-\frac{1}{\mu_2^2} + \frac {\sigma}{\mu_2^2(1+\frac y {B})^{\frac 32}}
, \ \ 0<\ti \psi'(y) = \frac {3\sigma}{2\mu_2^2B(1+\frac y {B})^{\frac 52}}< \frac {2\sigma} {B} .$$
Thus,
\begin{align*}
2 f_1^\sim +  & \int_{|y|<\frac {B}3}   \pbt' (\ti \eta_{yy})^2  
  + \frac{\mu_2} {B} \int_{|y|<\frac {B}3} \left( 3(\ti \eta_y)^2 + \ti \eta^2 - 5 Q^4 \ti \eta^2 +20 y Q' Q^3 \ti \eta^2\right) \\
& \lec (B^{-3}+\si+\delta(\al_0)) \int_{|y|<\frac {B}3} \left( (\ti \eta_y)^2 + \ti \eta^2\right)
\end{align*}
Recall  the following localized Virial type estimate
(see \cite[Lemma 3.4]{MMR1}). Note that $\ti\eta$ does not satisfy quite the same orthogonality conditions as in \cite{MMR1} (after rescaling) but \eqref{tietaforc} is sufficient (see also \cite[Proposition 4]{MMjmpa}). 

\medskip

{\sl There exists $B_0>100$ and $\theta>0$ such that for $B>B_0$,}
\begin{align}
&\int_{|y|<\frac {B}3} \left( 3(\ti \eta_y)^2 + \ti \eta^2 - 5 Q^4 \ti \eta^2 +20 y Q' Q^3 \ti \eta^2\right)
\nonumber\\& \geq \theta \int_{|y|\leq \frac {B}3} ((\ti \eta_y)^2 + \ti \eta^2) 
-\frac 1{\theta}\ip{\ti\eta}{Q}^2 -\frac 1{\theta}\ip{\ti\eta}{\vec\La Q}^2 
-\left(\frac 1{\theta B ^2}+\delta(\alpha_0)\right)\|\ti\eta\|_{L^2_{loc}}^2.\label{vir}
\end{align}

Taking $B>1$ large and then $\si>0$ small, related to the universal constant $\theta$, we obtain
\EQ{ 
 & 2 f_1^\sim + \int_{|y|< \frac {B}3} (\ti \eta_{yy})^2 \ti \psi'   + \frac \theta 2 \int_{|y|< \frac {B}3}   \ti \varphi' \left( (\ti \eta_y)^2 + \ti \eta^2\right) 
 \\&\ \lec (B^{-3} +  \delta(\alpha_0))\|\ti \eta\|_{L^2_{loc}}^2 + \ip{\ti \eta}{Q}^2 + |\ip{\ti\eta}{\vec\La Q}|^2
}
Combining this with \eqref{ygrand}, and choosing $B$ large enough and then $\alpha_0$ small, we obtain \eqref{F1}.
For more details, see step 3 of the proof of Proposition 3.1 in \cite{MMR1}.

\medskip

- Estimate of $f_2$. Since $\ip{\Lambda Q}{L\ti \eta} = -2\ip{Q}{\ti \eta}$, we have, after various integrations by parts, and using the definitions of $\pbt$ and $\vbt$,
\EQ{
  &\left| 2\ip{\ti \eta}{Q} + \ip{\Lambda Q}{-\pbt' \ti \eta_y + \pbt L \ti \eta + (\vbt-\pbt) \ti \eta } \right| 
  \\&\lec (B^{-1} + \sigma B^{\frac 12} +\delta(\al_0)) \ip{\ti \eta^2}{\vbt'}^{\frac 12}+B^{-1}|\ip{y\La Q}{\ti\eta}|.}
Similarly, using $LQ'=0$,
\EQ{
 &|\ip{\p_y Q}{-\pbt' \ti \eta_y + \pbt L \ti \eta + (\vbt-\pbt) \ti \eta}| 
 \\& \lec (B^{-1} + \sigma B^{\frac 12}  + \delta(\alpha_0)) \ip{\ti \eta^2}{\vbt'}^{\frac 12} + B^{-1}|\ip{\La Q-Q/2}{\ti\eta}|.}
Using \eqref{cl:p1}, we obtain, choosing $B$ large enough, and then $\sigma$ small enough, 
\begin{align*}
 & |\ip{\ti\Om}{\ti\eta}\cdot\ip{\vec\p Q}{-\pbt' \ti \eta_y + \pbt L \ti \eta + (\vbt-\pbt) \ti \eta}|
 \\&\leq C \|\ti\eta\|_{L^2_{loc}} \left[(B^{-1} + \sigma B^{\frac 12} + \delta(\al_0)) \ip{\ti \eta^2}{\vbt'}^{\frac 12} + |\ti a| +B^{-1}|\ip{\vec\La Q}{\ti\eta}|\right] 
 \\&\leq  \frac {\theta}{100} \ip{\ti \eta^2}{\vbt'} +  CB|\ti a|^2 + |\ip{\vec\La Q}{\ti\eta}|^2. 
\end{align*}

Using \eqref{cl:p1} and Cauchy-Schwarz inequality,
\EQ{  \label{eq:di}
 &\left|\ip{\Lambda \eta_1}{- (\pbt \ti \eta_y)_y + \vbt   \ti \eta - 5Q^4 \ti \eta \pbt} \right| 
 \\ &\lec \|\La\eta_1\|_{L^2_y\Lw{\frac{1}{3B}}{2}} \|- (\pbt \ti \eta_y)_y + \vbt   \ti \eta - 5Q^4 \ti \eta \pbt\|_{L^2_y\Lw{\frac{-1}{3B}}{-2}},}
where the first norm on the right is bounded by $\log B\|\eta_1\|_{H^1_y\Lw{\frac{1}{4B}}{\frac 52}}$, and the other norm is bounded by (using the decay order of $\ti\fy$ and $\ti\psi$) 
\EQ{ \label{fL2est}
 &\|- (\pbt \ti \eta_y)_y + \vbt   \ti \eta - 5Q^4 \ti \eta \pbt\|_{L^2_y\Lw{-\frac{\mu_2}{2B}}{\frac{-5}{4}}} 
 \\& \lec \||\ti\psi\ti\eta_{yy}|+|\ti\psi'\ti\eta_y|+|\ti\fy\ti\eta|\|_{L^2_y\Lw{-\frac{\mu_2}{2B}}{\frac{-5}{4}}}
 \lec \|\ti\eta_{yy}\|_{L^2_y\Lw{\frac 1B}{\frac{-5}{4}}} + \||\ti\eta_y|+|\ti\eta|\|_{L^2_y\Lw{\frac{\mu_2}{2B}}{\frac{-3}{4}}}
 \\& \lec B^{\frac 12}\ip{(\ti\eta_{yy})^2}{\ti\psi'}^{\frac 12} + B^{\frac 12}\ip{(\ti\eta_y)^2+\ti\eta^2}{\ti\fy'}^{\frac 12}.}
The term $\ip{\p_y \eta_1}{- (\pbt \ti \eta_y)_y + \vbt   \ti \eta - 5Q^4 \ti \eta \pbt}$ is estimated similarly (it is actually easier). 
Hence in conclusion, using Lemma \ref{le:2bis}, for $B$ large enough, we obtain
\begin{align}
 \left|\ip{\ti\Om}{\ti\eta}\cdot\ip{\vec\p \eta_1}{- (\pbt \ti \eta_y)_y + \vbt   \ti \eta - 5Q^4 \ti \eta \pbt} \right|
 &\le  \frac {\theta}{100} \ip{(\ti \eta_{yy})^2}{\pbt'} + \frac {\theta}{100}  \ip{ (\ti \eta_{y})^2 + \ti \eta^2}{\vbt'} \nonumber 
\\&\quad + CB^2 \|\ti \eta\|_{L^2_{loc}}^2 \|\eta_1\|_{H^1_y\Lw{\frac 1{4B}}{\frac 52}}^2. \label{eq:E1}
\end{align}

- Estimate of $f_3$. Integrating by parts, we   see that
\begin{align*}
& \ip{\Lambda \ti \eta}{- (\pbt \ti \eta_y)_y + \vbt   \ti \eta - 5Q^4 \ti \eta \pbt }
\\& = - \frac 12 \ip{y \pbt'}{\ti \eta_y^2} + \ip{\pbt}{(\ti\eta_y)^2}
- \frac 12 \ip{\ti\eta^2}{y \vbt'} - 5 \ip{ (2 \pbt'-y \pbt) Q^4}{\ti \eta^2} +  20 \ip{\pbt \Lambda Q Q^3}{\ti \eta^2}. 
\end{align*}
Thus,
\EQ{
& \ip{\Lambda \ti \eta}{- (\pbt \ti \eta_y)_y + \vbt   \ti \eta - 5Q^4 \ti \eta \pbt }
 + \frac 12 \ip{\ti\eta^2}{y \vbt'} - \ip{\pbt}{(\ti\eta_y)^2}
\\& = - \frac 12 \ip{y \pbt'}{(\ti \eta_y)^2} 
 - 5 \ip{(2 \pbt-y \pbt) Q^4}{\ti \eta^2} + 20 \ip{ \pbt \Lambda Q Q^3}{\ti \eta^2}.
 \label{anotherremark}}
Using   $|y\pbt'| \lesssim B \vbt'$ on $\RR$ and then Lemma \ref{le:2bis}, we obtain 
$$
|f_3|
 \lesssim  \left| {\Omega_2} \right| B \ip{(\ti \eta_y)^2 + \ti \eta^2}{ \vbt'}
 \lesssim  B \delta(\alpha_0) (1+s)^{-\frac 72} \ip{(\ti \eta_y)^2 + \ti \eta^2}{ \vbt'}.
$$ 

- Estimate of  $f_4$. Integrating by parts, and using the decay properties of $Q$,
\begin{align*}
& \left| Z_2 \ip{\ti\eta _y}{ -\pbt' \ti \eta_y + \pbt L \ti \eta + (\vbt-\pbt) \ti \eta} \right|   
\\& \leq \delta(\alpha_0)(1+s)^{-\frac 72} B \ip{(\ti \eta_y)^2   + \ti \eta^2}{\ti\varphi'} \leq \frac {\theta}{100}  \ip{(\ti \eta_{y})^2+\ti \eta^2}{\vbt'}  .
\end{align*}

- Estimate of $f_5$. Using \eqref{eq:eta2} and Lemma \ref{le:Sob}, we have pointwise bounds
\EQ{
 & \|\eta_j\|_{L^\I_y\Lw{0}{\frac 94}} \lec \|\eta_j\|_{L^2_y\Lw{0}{\frac 92}}^{\frac 12}\|\p_y\eta_j\|_{L^2_y}^{\frac 12}\le \delta(\al_0),
 \\& \|\ti\eta\|_{L^\I_y\Lw{\frac{\mu_2}{2B}}{\frac{-3}{4}}}
 \lec \|\ti\eta\|_{H^1_y\Lw{\frac{\mu_2}{2B}}{\frac{-3}{4}}}
 \lec B^{\frac 12}\ip{(\ti\eta_y)^2+\ti\eta^2}{\ti\fy'}^{\frac 12}.}
Hence, from \eqref{deftiR}, 
\EQ{ \label{bd tiR}
 \|\p_y(\ti R \ti\eta)\|_{L^2_y\Lw{\frac{\mu_2}{2B}}{6}} 
 &\lec \||Q_x|+Q+|\eta|\|_{L^\I_y\Lw{0}{\frac 94}}^3 \|\eta\|_{H^1_y}\|\ti\eta\|_{H^1_y\Lw{\frac{\mu_2}{2B}}{\frac{-3}{4}}}
 \\& \lec \al_0 B^{\frac 12} \ip{(\ti\eta_y)^2+\ti\eta^2}{\ti\fy'}^{\frac 12}.}
Hence, using \eqref{fL2est}
\EQ{
 |f_5| &\le \|\p_y\ti R\ti\eta\|_{L^2_y\Lw{\frac{\mu_2}{2B}}{6}} \|- (\pbt \ti \eta_y)_y  - 5Q^4 \ti \eta \pbt + \vbt \ti \eta \|_{L^2_y\Lw{\frac{-\mu_2}{2B}}{-6}}
 \\&\le C\al_0 B \left\{ \ip{(\ti\eta_{yy})^2}{\ti\psi'} + \ip{(\ti\eta_y)^2+\ti\eta^2}{\ti\fy'}\right\}
 \\&\le \frac{\theta}{100}\left\{ \ip{(\ti\eta_{yy})^2}{\ti\psi'} + \ip{(\ti\eta_y)^2+\ti\eta^2}{\ti\fy'}\right\},}
taking $\al_0>0$ small enough depending on $B$. 

Gathering these estimates for $f_1$--$f_5$, we obtain \eqref{dFt}.
This finishes the proof of Claim \ref{cl:e} and so that of Proposition \ref{pr:s4}.
\end{proof}


\section{$C^1$ regularity} \label{sec:5}
The following is a more precise version of Theorem \ref{th:manifold} about the regularity of the manifold constructed in the previous section. Put
\EQ{
 H^1_\perp :=\{\fy\in H^1(\RR) \mid \ip{\fy}{Q}=0\}.}
\begin{proposition}\label{pr:4}
There exists a map (the Fr\'echet derivative) $A':\calA_0^\perp\to(H^1_\perp)^*$ with the following properties: For any $\ga_0\in\calA_0^\perp$,
\[
  \forall \ga\in\calA_0^\perp,\quad 
    |A(\ga)-A(\ga_0)-A'(\ga_0)(\ga-\ga_0)| =o\left(\|\ga-\ga_0\|_{H^1}\right).
\]
Moreover, 
for any $\ga_0\in\calA_0^\perp$ and any $\e>0$, there exists $\delta>0$ such that 
\[  
    \ga\in\calA_0^\perp \text{ and }   \|\ga-\ga_0\|_{H^1}<\delta 
  \implies 
 \|A'(\ga)-A'(\ga_0)\|_{(H^1_\perp)^*} <\e,
\]
and there exists an absolute constant $C>0$ such that 
\EQ{
 \forall \ga\in\calA_0^\perp,
 \quad \|A'(\ga)\|_{(H^1_\perp)^*} \le C\|\ga\|_{H^1}.}
\end{proposition}
\begin{remark}
The above regularity of $A$ is weaker than $C^1$ in the normed space 
\EQ{
 X_1:= H^1_\perp \cap L^2(x_+^{10}dx), \quad \|f\|_{X_1}:=\|f\|_{H^1},} 
because the domain $\calA_0^\perp$ of the graph $A$ is not open in $X_1$. Indeed, as mentioned in Introduction, the uniform bound on $\|u(0)-Q\|_{L^2(x_+^{10}dx)}$ is crucial for the whole argument. 
However, it is stronger than $C^1$ in the Banach space (which is the statement in Theorem \ref{th:manifold})
\EQ{
 X_2:= H^1_\perp \cap L^2(x_+^{10}dx), \quad \|f\|_{X_2}:= \|f\|_{H^1} + \|f\|_{L^2(x_+^{10}dx)},}
since the $L^2(x_+^{10}dx)$ norm is not used except for the definition of the domain $\calA_0^\perp$. 
Actually, the $H^1$ norm in the estimates on $A'$ can be further weakened with a decaying weight similar to those in Proposition \ref{th:s4}. 
\end{remark}

\begin{remark}
We can also estimate the difference of $A'$ in a Lipschitz way, as well as higher order derivatives of $A$. However, the higher regularity of $A$ requires the stronger decay and regularity of the solution, because of the term $\La\eta$ in the equation (see, e.g., \eqref{eqeta'}). Since the stronger conditions require smaller $\al_0$, there is some limitation in this way of proving the regularity of $A$ for each fixed ball of $\al_0>0$ in $H^1$, even if we restrict $\eta(0)$ to $C_0^\I(\RR)$. 
\end{remark}

\begin{proof}[Proof of Proposition \ref{pr:4}]
\noindent {\bf Step 1.} {Existence of a limit of difference quotient.} 
Fix any $\ga\in\calA_0^\perp$ and let $\ga_n\in\calA_0^\perp$ be a sequence such that $\|\ga_n-\ga\|_{H^1}\to 0$ as $n\to\I$. 
Let $\eta$ and $\eta_n$ be the solutions of \eqref{eqetaauto} corresponding to the initial data $\eta(0)=A(\ga)Q+\ga$ and $\eta_n(0)=A(\ga_n)Q+\ga_n$ so that they are in the (Soliton) regime. 

Take any sequence $h_n>0$ such that  
\EQ{
 & \ti\ga_n:=\frac{\ga_n-\ga}{h_n}, \quad  N:=\sup_{n\in\NN}\|\ti\ga_n\|_{H^1}<+\infty,}
and put
\EQ{ \ti\eta_n:=\frac{\eta_n-\eta}{h_n}, 
 \quad \ti a_n:=\|Q\|_{L^2}^{-2}\ip{Q}{\ti\eta_n}
 =\frac{A(\gamma_n)-A(\gamma)}{h_n}. }
Passing to a subsequence (still denoted by $\gamma_n$), there is a weak limit
\EQ{
 \ti\ga_n \to \exists \ga' \text{ in }\weak{H^1}, \quad \|\ga'\|_{H^1}\le N, \quad \ip{\ga'}{Q}=0.}
$\ti\eta_n$ satisfies 
\begin{align}
 \p_s \ti\eta_n = \p_y[(L - \ti R(\eta_n,\eta)\ti\eta] + \ip{\ti \Omega(\eta_n,\eta)}{\ti\eta_n}\cdot \vec\p(Q+ \eta_n)
  + \vec\Om(\eta)\cdot\vec\p \ti \eta_n, \label{eqenn}
\end{align}
\EQ{ \label{tietaforcnn}
 (\p_s+1)\ip{\ti\eta_n}{\vec\La Q}=0,\quad 
 \liminf_{s\to\I}|\ip{\ti\eta_n}{Q}|=0.}
Hence Proposition \ref{pr:s4} yields
\EQ{ \label{fre}
 & |\ti a_n|  \lec \min({\delta (\alpha_0)} (1+s)^{-\frac 52},\|\ga\|_{H^1}) , 
 \quad \|\ti\eta_n\|_{L^\I_sH^1_y\Lw{\frac 1B}{0}\cap L^2_sH^1_y\Lw{\frac 1B}{\frac{-3}{4}}} \lec N. }

Using the above uniform bound, together with the $H^1\subset L^\I$ bound on $\eta_n$ and $\eta$ for $\ti R(\eta_n,\eta)$ (cf.~\eqref{bd tiR}), it is easy to see that for any $f\in C_0^\I(\RR)$, $\ip{\p_s\ti\eta_n}{f}$ is uniformly bounded in $s\ge 0$ and $n\in\NN$. 
Therefore by Ascoli-Arzela, for any countable set $D\subset C_0^\I(\RR)$, there exists a subsequence (still denoted by $\ga_n$) such that $\ip{\ti\eta_n(s)}{f}$ converges locally uniformly on $s\ge 0$ for every $f\in D$. 
Choosing $D \subset C_0^\I$ dense in $L^2\Lw{\frac{-1}{B}}{0}$, we deduce that $\ti\eta_n$ has a weak limit 
\EQ{
 \ti\eta_n \rightharpoonup \eta'}
in $C([0,\I);\weak{H^1_y\Lw{\frac{1}{B}}{0}})$ and $\weak{L^2_sH^1_y\Lw{\frac{1}{B}}{\frac{-3}{4}}}$. 
From \eqref{fre}, we have strong convergence $\eta_n\to\eta$ in $L^\I_sH^1_y\Lw{\frac{1}{B}}{0}\subset L^\I_sL^\I_y\Lw{\frac{1}{B}}{0}$, and so
\EQ{
 \ti R(\eta_n,\eta) \to R'(\eta)=\ti R(\eta,\eta)
 = \sum_{j=1}^4 \frac{5! Q^{4-j} \eta^j}{j!(4-j)!}} 
in $C([0,\I);L^\I\Lw{\frac{4}{B}}{0})$. 
Using these convergence in \eqref{deftiOm}, we obtain 
\EQ{
 \ti\Om(\eta_n,\eta) \to \vec\Om'(\eta)=\ti\Om(\eta,\eta)
 =  -M(\eta)^{-1}(1-L\p_y+\ti R(\eta,\eta)\p_y)\vec\La Q -  M(\eta)^{-1}M'\vec\Om(\eta)}
in $C([0,\I);e^{\frac{-|y|}{2}}L^\I_y)$. 
Also, $\ti a_n\to a'$ in $C([0,\I);\RR)\cap L^2((0,\I);\RR)$ where
\EQ{
 a':=\|Q\|_2^{-2}\ip{\eta'}{Q},}
which inherits the uniform bound from \eqref{fre}: for all $0\leq s<\infty$,
\EQ{ \label{bd a'}
 |a'(s)|  \lec \min({\delta (\alpha_0)} (1+s)^{-\frac 52},\|\ga\|_{H^1})N.}
Hence $\eta'$ satisfies the limit (linearized) equation in the distribution sense in $y\in\RR$, and the classical sense in $0<s<\I$:  
\EQ{ \label{eqeta'}
 \p_s\eta' &= \p_y[(L - R'(\eta))\eta']  + \ip{\vec\Om'(\eta)}{\eta'}\cdot\vec\p(Q+\eta) 
 + \vec\Om(\eta)\cdot\vec\p\eta' 
  \quad \text{ in }\calD'(\RR)}
with the initial data $\eta'(0)=a'(0)Q+\ga'$. 

\medskip

\noindent {\bf Step 2.} {Uniqueness of the weak limit at a fixed $u$.}  
The proof of $C^1$ follows from the following linear estimates for the above equation of $\eta'$. 
\begin{lemma} \label{le:ddiff}
Let $B>100$ large enough and then $\al_0>0$ small enough. Let $\ga\in\calA_0^\perp$ and  $\eta$ be the solution of \eqref{eqetaauto} for the initial data $\eta(0)=A(\ga)Q+\ga$. 
Let $\eta'\in C([0,\I);\weak{H^1_y\Lw{\frac{1}{B}}{0}})\cap L^2_sH^1_y\Lw{\frac{1}{B}}{\frac{-3}{4}}$ be a solution of \eqref{eqeta'}. Then we have  
\EQ{ \label{esteta'}
 & |\ip{Q}{\eta'}| \le C(B)\min(\|\ga\|_{H^1},\delta(\al_0)(1+s)^{-\frac 52})\|P_\perp\eta'(0)\|_{H^1\Lw{\frac 2B}{0}},
 \\& \|\eta'\|_{L^\I_sH^1_y\Lw{\frac 4B}{0} \cap L^2_s H^1_y\Lw{\frac 4B}{\frac{-3}{4}}} \le C(B)\|P_\perp\eta'(0)\|_{H^1_y\Lw{\frac 2B}{0}}.}
\end{lemma} 
\begin{proof}
As is already indicated above, \eqref{eqeta'} for $\eta'$ is the same as \eqref{eqe}, once we put $\eta_1=\eta_2=\eta$ and $\ti\eta=\eta'$. 
It also satisfies \eqref{tiadec} because $\eta'\in L^2_sH^1_y\Lw{\frac{1}{B}}{\frac{-3}{4}}$. 
Moreover, one can easily see that \eqref{eqeta'} implies 
\EQ{
 (\p_s+1)\ip{\vec\La Q}{\eta'}=0.}
Hence applying Proposition \ref{pr:s4} to $\eta'$ yields the conclusion. 

The only possible issue is that $\eta'$ solves the equation \eqref{eqeta'} only in the distribution sense. 
However, since the equation is linear and we have enough estimates on $R'(\eta)$, it is easy to see that the distribution solution of \eqref{eqeta'} is unique. Hence it suffices to prove the above estimates only for smooth initial data $\eta'(0)$, then we can use the integration by parts as in the proof of Proposition \ref{pr:s4}. See Appendix \ref{ap:weaksol} for more detail about the treatment of distribution solutions. 
\end{proof}

The above lemma says in particular that given $\gamma$, $\gamma'$ as above, there is at most one value of $a_0'$ such that the (unique) solution $\eta'$ of \eqref{esteta'} satisfies $\liminf_{t\to \infty} \ip{\eta'}{Q}=0$.
This rigidity  implies that the weak limit of $\ti\eta_n$ is uniquely determined by $\ga$ and $\ga'$. In particular, 
\EQ{
 A'(\ga)\ga':=a'(0)=\|Q\|_2^{-2}\ip{\eta'(0)}{Q}\in\RR }
 does not depend on the choice of a particular sequence $(\gamma_n)$ and 
is well-defined for any $\ga\in\calA_0^\perp$ and any $\ga'\in H^1_\perp$. The map $A'(\ga)$ is linear for $\ga'$ because the equation \eqref{eqeta'} is linear for $\eta'$. Moreover, it is bounded by \eqref{esteta'}
\EQ{
 |A'(\ga)\ga'|=|a'(0)| \lec \|\ga\|_{H^1}\|\ga'\|_{H^1}
 \implies \|A'(\ga)\|_{(H^1_\perp)^*} \lec \|\ga\|_{H^1}.}
As a consequence, for any sequence $\calA_0^\perp\ni\ga_n\to\ga$ strongly in $H^1$ such that $\ti\ga_n=(\ga_n-\ga)/\|\ga_n-\ga\|_{H^1}\to\ga'$ weakly in $H^1$, we have,  for a subsequence
\EQ{ \label{Frechet}
 A'(\ga)\ti\ga_n \to A'(\ga)\ga'=a'(0).}
By standard arguments, this implies the differentiability of $A$, namely
\EQ{
 A(\ga_n)=A(\ga)+A'(\ga)(\ga_n-\ga)+o(\|\ga_n-\ga\|_{H^1}).}

\medskip

\noindent {\bf Step 3.} {Continuity of the derivative.} 
Let $\gamma_0\in\calA_0^\perp$ and let $\eta_0$ be the solution of \eqref{eqetaauto} with $\eta_0(0)=A(\ga_0)Q+\ga_0$. 
To show the continuity of $A'$ in $(H^1_\perp)^*$ at $\gamma_0$, take any sequence $\ga_n\in\calA_0^\perp$ converging to $\ga_0$ strongly in $H^1$, and any sequence $\ga_n'\in H^1_\perp$ satisfying $\|\ga_n'\|_{H^1}\le 1$ and converging to $\ga_0'$ weakly in $H^1$. 
Let $\eta_n$ be the solution of \eqref{eqetaauto} with the initial data $\eta_n(0)=A(\ga_n)Q+\ga_n$. 
Let $\eta'_n$ be the solution of \eqref{eqeta'} with $\eta=\eta_n$ and the initial data $\eta'_n(0)=A'(\ga_n)\ga_n'Q+\ga_n'$. 
Applying again Proposition \ref{pr:s4}, \eqref{eq:s4tri} implies 
\EQ{
 \|\eta_n-\eta_0\|_{L^\I_sH^1_y\Lw{\frac 4B}{0}\cap L^2_sH^1_y\Lw{\frac 4B}{\frac{-3}{4}}} \lec \|\ga_n-\ga_0\|_{H^1} \to 0, }
and \eqref{esteta'} implies 
\EQ{
 \|\eta'_n\|_{L^\I_sH^1_y\Lw{\frac 4B}{0}\cap L^2_sH^1_y\Lw{\frac 4B}{\frac{-3}{4}}} \lec \|\ga_n'\|_{H^1} \le 1.}
Then the same argument as in Step 1 yields a weak limit $\eta'_\I$ after extracting a subsequence: 
\EQ{
 \eta'_n \rightharpoonup \eta'_\I}
in $C([0,\I);\weak{H^1_y\Lw{\frac 4B}{0}})\cap\weak{L^2_sH^1_y\Lw{\frac 4B}{\frac{-3}{4}}}$. 
Moreover, $\eta'_\I$ is a weak solution of \eqref{eqeta'} with $\eta=\eta_0$ and the initial data $\eta'_\I(0)=a'_\I Q+\ga_0'$, where 
\EQ{
 a'_\I:=\lim_{n\to\I}A'(\ga_n)\ga'_n.}
Then, Lemma \ref{le:ddiff} implies 
\EQ{
 a'_\I = A'(\ga_0)\ga'_0,}
and so $\lim_{n\to \I}\|A'(\ga_n)-A'(\ga_0)\|_{(H^1_\perp)^*}= 0$ and the continuity of $A'$ at $\ga_0$ is proved.
\end{proof}

\appendix
\section{Weak solutions of the linearized gKdV equation} \label{ap:weaksol}
Here we prove uniqueness and regularity of weak or distributional solutions of \eqref{eqeta'}. 
Using the Sobolev bounds in Lemma \ref{le:2bis} for $\eta$ together with the weighted Sobolev inequality (Lemma \ref{le:Sob}) as before, it is easy to see that the distribution solution obtained in Section \ref{sec:5} is in the setting of the following lemma. Let $\calL(X,Y)$ denote the Banach space of bounded linear operators from $X$ to $Y$. 
\begin{lemma} \label{le:weaksol}
Let $T,\si,\nu>0$, $D>1$ and $\NN\ni k\ge 1$. Let $(\Om,Z) \in C([0,T];\RR^2)$,  $B\in C([0,T];\calL(L^2\Lw{D^2(\si+k\nu)}{0},H^k\Lw{\si}{0}))$ and $m\in C([0,T];L^\I\Lw{0}{\frac{1}{2}+\e}\cap H^k\Lw{\nu}{\frac{1}{2}+\e})$ for some $\e>0$. 
Suppose that $|\int_0^s\Om(s')ds'|\le\log D$ for all $0<s<T$ and that $\xi \in C([0,T];\weak{L^2_y\Lw{\si}{0}})$ solves the following equation for $0<s<T$. 
\EQ{
 \p_s \xi = -\p_y(\p_y^2+m)\xi + (\Om\La+Z\p_y)\xi + B\xi  \quad \text{in }\calD'(\RR).}
\begin{enumerate}
\item If $\xi(0)=0$ then $\xi(s)=0$ for all $0\le s\le T$. 
\item If $\xi(0)\in H^k_y\Lw{\si}{0}$ then $\xi\in C([0,T];H^k_y\Lw{D^2(\si+k\nu)}{0})$. 
\end{enumerate}
\end{lemma}
Using this lemma with $k=1$, we obtain the uniqueness of the weak solution $\eta'$ of \eqref{eqeta'} considered in Section \ref{sec:5}. 
In order to justify the integration by parts needed in the proof of Lemma \ref{le:ddiff}, we start from arbitrary small $s=s_0>0$ and consider the case $\eta'(s_0)\in H^\I$. Since $\eta$ is regular enough for $s\ge s_0$ by \eqref{regeta}, we can use the above lemma for $k\le 9$, then $\eta'$ solves the equation in the classical sense, so that we can integrate by parts. By the density argument, the uniform estimates are extended to the general case $\eta'(s_0)\in H^1_y\Lw{\frac 1B}{0}$. 
For $\al_0$ small enough (depending on $B$), we can take $D>1$ and $\nu>0$ small enough such that $D^2(1/B+\nu)<2/B$. Then $\eta'(s)$ is strongly continuous in $H^1_y\Lw{\frac 2B}{0}$, so that we can take the limit $s_0\to+0$, concluding the estimates by the initial data in Lemma \ref{le:ddiff}. 

\begin{proof}[Proof of Lemma \ref{le:weaksol}] 
First we use the change of variables $(s,y)\mapsto(t,x)$ as in \eqref{tx2sy}, with $\mu(s):=\int_0^s\Om(s')ds'$ and $z(s):=\int_0^s\mu(s')(Z(s')+1)ds'$. Put $v(t,x)=\xi(s,y)$. Then $v$ solves 
\EQ{
 \p_t v = -\p_x(\p_x^2+\hat{m})v+\hat{B}v \quad\text{ in }\calD'(\RR)}
for $0<t<T':=t(T)$, where 
\EQ{
 \hat{m}(t,x):=\mu(s)^2 m(s,y), \quad \hat{B}(t):=\mu(s)^3 \calT(s)B(s)\calT(s)^{-1},}
and $\calT(s)$ is the operator of the transform $(\calT(s)\fy)(y)=\fy(x)$. By the assumption, we have $1/D\le \mu\le D$ for all $0\le s\le T$, and so, $v \in C([0,T'];\weak{L^2_y\Lw{D\si}{0}})$, $\hat{m}\in C([0,T'];L^\I\Lw{0}{\frac{1}{2}+\e}\cap H^k\Lw{D\nu}{\frac{1}{2}+\e})$ and $\hat{B}\in C([0,T'];\calL(L^2\Lw{D(\si+k\nu)}{0}),H^k\Lw{D\si}{0}))$. 

Second we use the standard mollifier argument. Choose any $\de\in C_0^\I(\RR)$ satisfying $\supp\de\subset(-1/2,1/2)$ and $\ip{\de}{1}=1$. Let $\de_n:=n\de(nx)$ and $v_n:=\de_n*v$. Then for all $l\ge 0$, $v_n\in C^1([0,T'];H^l\Lw{D\si}{0})$  solves in the classical sense
\EQ{ \label{eqvn}
 \p_t v_n = -\p_x^3 v_n  + \de_n*(\hat{B}v) - \p_x\de_n*(\hat{m}v),}
which implies the $C^1$ regularity in $t$ of $v_n$. 

Finally, we use the weighted $L^2$ estimate as before. Fix $\e>0$ so small that we can use the bound on $\hat{B}$, and choose $w\in C^\I(\RR)$ such that 
\EQ{
 & w^{(l)}(x) \sim e^{2D\si x} \quad \text{ for $x<0$ and $l=0,1,2,3$},
 \\& w'(x)\sim (1+x)^{-1-2\e}, \quad |w^{(l)}|(x)\lec 1 \quad\text{ for $x>0$ and $l=0,1,2,3$}.}
From \eqref{eqvn} we have a weighted $L^2$ identity 
\EQ{
 \p_t\ip{v_n^2}{w} = -3\ip{(v'_n)^2}{w'} + \ip{v_n^2}{w'''}
  + 2\ip{\de_n*(\hat{B}v)}{v_n w} + 2\ip{\de_n*(\hat{m}v)}{(v_nw)'}.}
Using the bounds on $\hat{B}$, $\hat{m}$, $w$ and $w'$, we have 
\EQ{
 & \|\de_n*(\hat{B}v)\|_{L^2_x\Lw{D\si}{0}}
  \lec \|\hat{B}v\|_{L^2_x\Lw{D\si}{0}} \lec \|v\|_{L^2_x\Lw{D\si}{0}},
 \\& \|\de_n*(\hat{m}v)\|_{L^2_x\Lw{D\si}{\frac{1}{2}+\e}}
  \lec \|\hat{m}\|_{L^\I_x\Lw{0}{\frac{1}{2}+\e}}\|v\|_{L^2_x\Lw{D\si}{0}},
 \\& \|(v_nw)'\|_{L^2_x\Lw{-D\si}{\frac{-1}{2}-\e}}
 \lec \|v_n'\|_{L^2_x\Lw{D\si}{\frac{-1}{2}-\e}}+\|v_n\|_{L^2_x\Lw{D\si}{0}}. }
Thus using Cauchy-Schwarz and the bounds on $w'$ and $w'''$, we obtain 
\EQ{
 \p_t\ip{v_n^2}{w} \le -2\ip{(v'_n)^2}{w'} + C \ip{v_n^2}{w} + C \ip{(v-v_n)^2}{w}.}
Hence integrating in $0<t<T'$, 
\EQ{ 
 \ip{v_n^2(t)}{w} \le e^{Ct}\ip{v_n^2(0)}{w} + \int_0^t e^{C(t-t')}C \ip{(v-v_n)^2(t')}{w}dt'.}
As $n\to\I$, the last term is vanishing by the dominated convergence in $t'$. Thus we obtain 
\EQ{
 \ip{v^2(t)}{w} \le e^{Ct}\ip{v^2(0)}{w}.}
In particular, if $\xi(0)=0$ then $v(0)=0$ and so $v(t)=0$ for all $0\le t\le T'$. We can apply the same argument to the difference $v_n-v_m$, which yields
\EQ{
 \ip{(v_n-v_m)^2(t)}{w} \le e^{Ct}\ip{(v_n-v_m)^2(0)}{w}
 + \int_0^t Ce^{C(t-t')}F_{n,m}^2(t')dt',}
where $F_{n,m}:=\|v_n-v_m\|_{L^2\Lw{D\si}{0}}+\|(\de_n-\de_m)*(\hat{B}v)\|_{L^2\Lw{D\si}{0}}
 +\|(\de_n-\de_m)*(\hat{m}v)\|_{L^2\Lw{D\si}{\frac 12+\e}}$. 
As $n,m\to\I$, the right hand side is vanishing by the dominated convergence, uniformly for $0\le t\le T'$. 
Hence the limit $v$ is also strongly continuous in $L^2_x\Lw{D\si}{0}$. 
The same argument can be applied to the derivatives for $1\le j\le k$ 
\EQ{ \label{eqvnj}
 \p_t v^{(j)}_n = -\p_x^3 v^{(j)}_n + \de_n*\p_x^j(\hat{B}v) - \p_x\de_n*\p_x^j(\hat{m}v).}
The only difference from the case $j=0$ appears in the last term, for which we have 
\EQ{ \label{deri mv}
 &\|\p_x^j(\hat{m}v)\|_{L^2_x\Lw{D(\si+j\nu)}{\frac{1}{2}+\e}} 
 \\&\ \lec \|\hat{m}\|_{L^\I_x\Lw{0}{\frac12+\e}}\|v^{(j)}\|_{L^2_x\Lw{D(\si+j\nu)}{0}}+ \sum_{l=0}^{j-1}\|\hat{m}^{(j-l)}\|_{H^1_x\Lw{D\nu}{\frac{1}{2}+\e}}\|v^{(l)}\|_{L^2_x\Lw{D(\si+l\nu)}{0}}, }
where we used the weighted $L^\I$ Sobolev, see Lemma \ref{le:Sob}. 
Hence inductively for each $j$, after modifying the weight function $w$ such that 
\EQ{
  w^{(l)}(x) \sim e^{2D(\si+j\nu)x} \quad \text{for $x<0$ and $l=0,1,2,3$,}}
we obtain by the same argument as for $j=0$, using the induction hypothesis for the last term of \eqref{deri mv}, 
\EQ{
 \|v^{(j)}(t)\|_{L^2_x\Lw{D(\si+j\nu)}{0}}
  \lec e^{Ct}\|v^{(j)}(0)\|_{L^2_x\Lw{D(\si+j\nu)}{0}} \lec e^{Ct} \|\xi^{(j)}(0)\|_{L^2_y\Lw{\si+j\nu}{0}}.}
Similarly, we obtain uniform convergence of $v^{(j)}_n$ as $n\to\I$ in $L^2_x\Lw{D(\si+j\nu)}{0}$. Changing back the variables $(s,y)\mapsto(t,x)$, we obtain $\xi\in C([0,T];H^k\Lw{D^2(\si+k\nu)}{0})$. 
\end{proof}

\end{document}